\documentclass[11pt,a4paper]{amsart}
\usepackage{amssymb,amsmath,amsthm}

\usepackage[pdftex]{graphicx}

\usepackage[colorlinks=true, pdfstartview=FitV, linkcolor=blue, citecolor=blue, urlcolor=blue,pagebackref=false]{hyperref}

\usepackage{esint}
\usepackage{mathrsfs}

\usepackage{times,latexsym,color}

\newcommand{\BOX}{\ensuremath\Box}

\newtheorem{theorem}{Theorem}
\newtheorem*{theorem*}{Theorem}
\newtheorem{proposition}{Proposition}
\newtheorem{lemma}[proposition]{Lemma}

\theoremstyle{remark}
\newtheorem{remark}[proposition]{Remark}

\theoremstyle{definition}
\newtheorem{definition}[proposition]{Definition}

\DeclareMathOperator{\Rel}{Re}

\newcommand{\N}{\mathbb{N}}

\newcommand{\C}{\mathbb{C}}
\newcommand{\R}{\mathbb{R}}

\newcommand{\ep}{\varepsilon}

\definecolor{darkgreen}{rgb}{0,0.5,0}
\definecolor{darkblue}{rgb}{0,0,0.7}
\definecolor{darkred}{rgb}{0.9,0.1,0.1}
\definecolor{lightblue}{rgb}{0,0.51,1}

\begin{document}

\title[Scale-invariant estimates and vorticity alignment in the half-space]{Scale-invariant estimates and vorticity alignment for Navier-Stokes in the half-space with no-slip boundary conditions}

\author[T. Barker]{Tobias Barker}
\address[T. Barker]{DMA, \'{E}cole Normale Sup\'erieure, CNRS, PSL Research University, 75\, 005 Paris}
\email{tobiasbarker5@gmail.com}

\author[C. Prange]{Christophe Prange}
\address[C. Prange]{Universit\'e de Bordeaux, CNRS, UMR [5251], IMB, Bordeaux, France}
\email{christophe.prange@math.u-bordeaux.fr}

\subjclass[2010]{}
\date{\today}

\maketitle

\noindent {\bf Abstract} This paper is concerned with geometric regularity criteria for the Navier-Stokes equations in $\mathbb{R}^3_{+}\times (0,T)$ with no-slip boundary condition, with the assumption that the solution satisfies the `ODE blow-up rate' Type I condition.
More precisely, we prove that if the vorticity direction is uniformly continuous on subsets of 
$$\bigcup_{t\in(T-1,T)} \big(B(0,R)\cap\mathbb{R}^3_{+}\big)\times {\{t\}},\,\,\,\,\,\, R=O(\sqrt{T-t})$$
where the vorticity has large magnitude, then $(0,T)$ is a regular point. This result is inspired by and improves the regularity criteria given by Giga, Hsu and Maekawa in \cite{giga2014liouville}. We also obtain new local versions for suitable weak solutions near the flat boundary.
Our method hinges on new scaled Morrey estimates, blow-up and compactness arguments and `persistence of singularites' on the flat boundary. The scaled Morrey estimates seem to be of independent interest.

\vspace{0.3cm}


\vspace{0.3cm}


\section{Introduction}
This paper concerns the Navier-Stokes system in the half-space $\mathbb{R}^3_{+}:=\{(x',x_{3})\in\mathbb{R}^3,\,\,\,\,x_{3}>0\}$ with {no-slip} boundary condition:
\begin{equation}\label{noslip}
\partial_{t}u-\Delta u+u\cdot\nabla u+\nabla p=0,\,\,\,\,u(x,0)=u_{0}\,\,\,\,\,\nabla\cdot u=0,\,\,\,u|_{\partial\mathbb{R}^3_{+}}=0,\,\,\,\,x\in\mathbb{R}^3_{+}\,\,\,\,\,t>0.
\end{equation}
For the simplest case of the fluid occupying the whole-space $\mathbb{R}^3$, it is a Millennium prize problem \cite{fefferman2006existence} as to whether or not solutions to the Navier-Stokes equations, with Schwartz class divergence-free initial data, remain smooth for all times. 

In \cite{constantin1993direction}, Constantin and Fefferman provided a geometric regularity criteria for the vorticity  $\omega= \nabla\times u$ of solutions on the whole-space, which remarkably does not depend on scale-invariant quantities.\footnote{By scale-invariant quantities, we mean quantities which are invariant with respect to the Navier-Stokes rescaling $(u_{\lambda}(y,s),p_{\lambda}(y,s))=(\lambda u(\lambda y, \lambda^2 s), \lambda^2 p(\lambda y, \lambda^2 s))$. The vast majority of regularity criteria for the Navier-Stokes equations are stated in terms of scale-invariant quantities since, heuristically at least, the diffusive effects and non-linear effects are `balanced'.} Specifically, they showed that if\footnote{ Here $\angle(a,b)$ denotes the angle between the vectors $a$ and $b$.}
\begin{multline}\label{constantinfefferman}
|\sin(\angle( \omega(x+y,t), \omega(x,t))|\leq C|y|\\
\,\,\,\,\textrm{for}\,\,\,\,(x,t)\in \Omega_{d}:=\{(y,s)\in\mathbb{R}^3\times (0,T): |\Omega(y,s)|>d\}
\end{multline}
then $u$ is smooth on $\mathbb{R}^3\times (0,T]$. Their proof uses energy estimates for the vorticity equation
\begin{equation}\label{vorticity}
\partial_{t}\omega-\Delta\omega+u\cdot\nabla\omega-\omega\cdot\nabla u=0
\end{equation}
to get
\begin{equation}\label{vortenergyest}
\|\omega(\cdot,t)\|_{L_{2}(\mathbb{R}^3)}^2+ 2\int\limits_{0}^{t}\int\limits_{\mathbb{R}^3} |\nabla w|^2 dxdt\leq 2\int\limits_{0}^{t}\int\limits_{\mathbb{R}^3} (\omega\cdot\nabla u)\cdot\omega dxdt
\end{equation}
 Then Constantin and Fefferman proceed with a careful analysis of the stretching term
$$\int\limits_{0}^{t}\int\limits_{\mathbb{R}^3} (\omega\cdot\nabla u)\cdot\omega dxdt.$$
In particular, the Biot-Savart law, integration by parts and linear algebra identities are used to show that the most singular contribution of the stretching term can be controlled by
\begin{equation}\label{vorticitydirectioncontrol}
\int\limits_{0}^{t}\iint\limits_{\Omega_{d}\times\Omega_{d}} \frac{|\omega(x,s)|^2|\omega(y,s)||\sin(\angle{(\omega(y,s), \omega(x,s)})|}{|x-y|^3} dydxds.
\end{equation}
Crucially \eqref{constantinfefferman} depletes the singularity in the integral of this stretching term. For certain extensions of this geometric regularity criterion, we refer (non-exhaustively) to \cite{da2002regularizing,grujic2004interpolation,zhou2005new,grujic2006space,grujic2009localization,GM11,li2017vortex}.

For the case of the Navier-Stokes equations in $\mathbb{R}^3_{+}$ with no-slip boundary conditions \eqref{noslip}, vorticity is generated at the boundary. Specifically, $\omega$ has a non-zero trace on $\partial \mathbb{R}^3_{+}$. This provides a block to applying energy methods to the vorticity equation \eqref{vorticity}. Consequently, the following remains open
\begin{itemize}
\item[] \textbf{(Q)}: For finite energy solutions of the Navier-Stokes equations in $\mathbb{R}^3_{+}\times (0,T)$, with no-slip boundary condition and divergence-free initial data $u_{0}\in C^{\infty}_{0}(\mathbb{R}^3_{+})$, does \eqref{constantinfefferman} imply that $u$ is smooth on $\mathbb{R}^3_{+}\times(0,T]$?
\end{itemize} 
Note that with perfect slip boundary conditions
\begin{equation}\label{slip}
u_{3}= \partial_{3}u_{2}= \partial_{2} u_{1}\,\,\,\,\,\textrm{on}\,\,\,\,\,\,\partial{\mathbb{R}}^{3}_{+}, \qquad  u=(u_{1},u_{2},u_{3}),
\end{equation}
vorticity is not created at the boundary and the geometric regularity condition is known to hold \cite{da2006vorticity}.
The issue of vorticity creation at the boundary has meant that certain other situations that are understood for the Navier-Stokes equations in the whole-space $\mathbb{R}^3$, remain open for the case of the fluid occupying the half-space with no-slip boundary condition.
Two such examples are:
\begin{enumerate}
\item Existence of backward self-similar solutions  $$u(x,t)= \frac{1}{\sqrt{T-t}}U\Big(\frac{x}{\sqrt{T-t}}\Big),\,\,\,\,(x,t)\in\mathbb{R}^3_{+}\times (0,T)$$ 
 with finite local energy and no-slip boundary condition. For the whole-space such solutions were shown not to exist in \cite{tsai1998leray}.
\item Smoothness of solutions to the Navier-Stokes equations in $\mathbb{R}^3_{+}\times (0,\infty)$  with no-slip boundary condition that are axisymmetric without swirl. Such solutions have the following form in cylindrical coordinates
$$
u(r,z,t)= u_{r}(r,z)\vec{e}_{r}+u_{z}(r,z)\vec{e}_{z},
$$
where $\vec{e}_{r}=(\cos(\theta), \sin(\theta),0)$, $\theta\in(-\pi,\pi)$ and $\vec{e}_{z}=(0,0,1).$ 
For the whole-space, such solutions were shown to be smooth in \cite{lady}.
\end{enumerate}
To the best of our knowledge, the only  vorticity alignment regularity criteria known for the Navier-Stokes equations in $\mathbb{R}^3_{+}\times (0,T)$ (with no-slip boundary condition), was proven by Giga, Hsu and Maekawa in \cite{giga2014liouville} for mild solutions under the additional assumption that $u$ satisfies the Type I condition \eqref{TypeIhalf} (see below in Theorem \ref{CAnconcentratingglobal}). 
In particular, in \cite{giga2014liouville} it was shown that if $\eta:[0,\infty)\rightarrow \mathbb{R}$ is a nondecreasing continuous function with $\eta(0)=0$ , $\xi=\frac{\omega}{|\omega|}$ and $\Omega_{d}:=\{(x,t)\in\mathbb{R}^3_{+}\times(0,T): |\omega(x,t)|>d\}$ then the following holds true. Namely, the assumption
\begin{equation}\label{GigaHsuMaekawa}
|\xi(x,t)-\xi(y,t)|\leq \eta(|x-y|)\,\,\,\,\,\textrm{for}\,\,\textrm{all}\,\, (x,t),(y,t)\in\Omega_{d}
\end{equation}
implies that $u$ is bounded up to $t=T$. 

The main goal of our paper is to prove a vorticity alignment regularity criteria, for the Navier-Stokes equations in $\mathbb{R}^3_{+}\times (0,T)$ with no-slip boundary condition, that improves upon the result in \cite{giga2014liouville}.  In what follows, we say that $(x_{0},T)$ is a `regular point' of $u$ if there exists $r>0$ such that $u\in L^{\infty}((B(x_{0},r)\cap\mathbb{R}^3_{+})\times (T-r^2,T))$. Any point that is not regular is defined to be a `singular point'.  Let us state our main theorem below.
\begin{theorem}\label{CAnconcentratingglobal}
Suppose $u$ is a mild solution to the Navier-Stokes equations in $\mathbb{R}^3_{+}\times (0,T)$, $(T>1)$, with no-slip boundary condition $u|_{\partial \mathbb{R}^3_{+}}=0$ and with divergence-free initial data $u_{0}\in C^{\infty}_{0}(\mathbb{R}^{3}_{+}).$ Furthermore, suppose that for $(x,t)\in \mathbb{R}^3_{+}\times (0,T)$:
\begin{equation}\label{TypeIhalf}
|u(x,t)|\leq \frac{M}{\sqrt{T-t}}.
\end{equation}
Let $t^{(n)}\uparrow T$, $x_{0}\in\mathbb{R}^3_{+}\cup \partial \mathbb{R}^3_{+}$, $d>0 $ and $\delta>0$. Define $\omega=\nabla\times u$, $\xi:= \frac{\omega}{|\omega|}$, the set of high vorticity,
$$\Omega_{d}:=\{(x,t)\in \mathbb{R}^3_{+}\times(0,T): |\omega(x,t)|>d\},$$
the cone
$$ \mathcal{C}_{\delta, x_{0}}:=\bigcup_{t\in(T-1,T)}\{x\in\mathbb{R}^3_{+}: |x-x_{0}|<\delta\sqrt{T-t_{0}}\}  $$
and the time-sliced cone
$$ \mathcal{C}_{t^{(n)}, \delta, x_{0}}:= \bigcup_{n}\{(x,t^{(n)}):x\in\mathbb{R}^3_{+}\,\,\,\textrm{and}\,\,\, |x-x_{0}|< \delta \sqrt{T-t^{(n)}}\}.$$
Let $\eta: \mathbb{R} \rightarrow \mathbb{R}$ be a continuous function with $\eta(0)=0$.\\
Under the above assumptions, there exists $\delta(M,u_{0})>0$ such that the following holds true: on the one hand
\begin{equation}
\sup_{(x,t)\in \mathcal{C}_{t^{(n)}, \delta, x_0}} |\omega(x,t)|<\infty\Rightarrow\,\,\,\,\,\,(x_{0}, T)\quad\mbox{is a regular point of}\quad  u,
\label{e.vortbddreg}
\end{equation}
and on the other hand
\begin{align}\label{CAconditionconcentrate}\tag{CA}
|\xi(x,t)-\xi(y,t)|\leq \eta(|x&-y|)\,\,\,\,\,\,\,\,\,\,\,\,\,\,\,\textrm{in}\,\,\,\,\,\Omega_{d}\cap \mathcal{C}_{ \delta, x_0}\\
&\Downarrow\notag\\
(x_{0}, T)\quad \mbox{is a}\ & \mbox{regular point of}\quad u.\notag
\end{align}
\end{theorem}

The condition \eqref{CAconditionconcentrate} is dubbed the `continuous alignment condition'.

\begin{remark}\label{localisations}
In this paper we also obtain local variations of Theorem \ref{CAnconcentratingglobal}. To the best of the authors' knowledge those are the first local results regarding regularity under the vorti\-city alignment condition, for a Navier-Stokes solution having no-slip on the flat part of the boundary. Previously all localisations were known for only interior cases \cite{grujic2006space,grujic2009localization,GM11}. For precise statements of our local results, we refer the reader to Section \ref{sec.local} `Geometric regularity criteria near a flat boundary: local setting'.
\end{remark}
\begin{remark}\label{converse}
When $x_{0}$ is a regular point of $u$ belonging to the interior of the half-space, it can be seen that $u$ has H\"{o}lder continuous spatial derivatives near $x_{0}$ and the converse statement to Theorem \ref{CAnconcentratingglobal} is true, i.e. the vorticity alignment condition \eqref{CAconditionconcentrate} holds.
It is not clear to the authors whether or not that remains to be the case when $x_{0}$ lies on $\partial\mathbb{R}^3_{+}$. The difficulty is that there are examples in \cite{Kang05} and \cite{SSv10}, for a Navier-Stokes solution in the local setting with no-slip boundary condition on the flat part of the boundary, that demonstrate that boundedness of $u$ near the flat boundary does not imply boundedness of $\nabla u$. It is interesting to note that these examples do not provide a counterexample to the converse statement of Theorem \ref{CAnconcentratingglobal}. In fact in \cite{SSv10} the construction is based upon a monotone shear flow, whose vorticity direction is constant.
\end{remark}

\subsection{Comparison of Theorem \ref{CAnconcentratingglobal} to previous literature and novelty of our results}

In \cite{giga2014liouville} progress was made at establishing the geometric regularity criteria for the vorticity when the fluid occupies the half-space $\mathbb{R}^3_{+}$ with no-slip boundary condition, with the additional assumption that the solution satisfies Type I bounds.
As we mention above, the geometric regularity criterion remains in general an outstanding open problem when the fluid occupies the half-space $\mathbb{R}^3_{+}$ with no-slip boundary condition. Let us now state the result first proven in \cite{giga2014liouville}.
\begin{theorem*}[{\cite[Theorem 1.3]{giga2014liouville}}]\label{Gigaetal}
Suppose $u$ is a mild solution to the Navier-Stokes equations in $\mathbb{R}^3_{+}\times (0,T)$ with boundary condition $u|_{\partial \mathbb{R}^3_{+}}=0$ and with divergence-free initial data $u_{0}\in C^{\infty}_{0}(\mathbb{R}^{3}_{+})$. 
Let $d$ be a positive number and let $\eta$ be a nondecreasing\footnote{In \cite{giga2014liouville} the authors assume that $\eta$ is nondecreasing. However, this assumption can be removed. What matters is that $\eta$ is continuous and $\eta(0)=0$.} 
continuous function on $[0,\infty)$ satisfying $\eta(0)=0$. Assume that $\eta$ is a modulus of continuity in the $x$ variables for the vorticity direction $\xi=\omega/|\omega|$, in the sense that
\begin{equation}\label{vorticitydirection}
|\xi(t,x)-\xi(t,y)|\leq \eta(|x-y|)\,\,\,\,\,\,\textrm{for}\,\,\textrm{all}\,\,\,\,\,(x,t),(y,t)\in \Omega_{d},
\end{equation}
where $\Omega_{d}$ is defined in Theorem \ref{CAnconcentratingglobal}. Then $u$ is bounded up to $t=T$.
\end{theorem*}

The proof in \cite{giga2014liouville} relies on two steps:
\begin{enumerate}
\item Using the assumption \eqref{vorticitydirection} and a choice of suitable rescalings for the equation (blow-up procedure) to reduce to a non zero two-dimensional nontrivial bounded limit solution with positive vorticity, which is defined on either the whole-space or the half-space. This limiting solution to the Navier-Stokes equations belongs to one of two  special classes called `whole-space mild bounded ancient solutions' or `half-space mild bounded ancient solutions'. We refer the reader to subsection \ref{sec.mbas} for the definitions and discussion of these objects. 
\item Obtaining a contradiction by showing the limit function must be zero. This is done in \cite{GM11} and \cite{giga2014liouville} by proving a Liouville theorem for two-dimensional  whole-space and half-space mild bounded ancient solutions having positive vorticity and satisfying the bound $$\sup_{t<0} (-t)^{\frac{1}{2}} \|u(\cdot,t)\|_{L_{\infty}}<\infty.$$
\end{enumerate}
The rescaling used in the first step is very specific. Let $t^{(k)}\uparrow T$, $y^{(k)}$ and $R_{k}\downarrow 0$ be selected such that
\begin{equation}\label{e.mbasrescaling}
\frac{1}{R^{(k)}}-1\leq |u(y^{(k)}, t^{(k)})|\leq\sup_{0\leq s\leq t^{(k)}}\|u(\cdot,s)\|_{L_{\infty}(\mathbb{R}^3)}:= \frac{1}{R^{(k)}}.
\end{equation}
Then $u^{(k)}(y,t)=R^{(k)}u\big(R^{(k)}(y+y^{(k)}), (R^{(k)})^2(t+t^{(k)})\big)$ produces either a whole-space or half-space mild bounded ancient solution $u^{(\infty)}$ in the limit. By \cite[Theorem 1.3]{barker2015ancient}, $u^{(\infty)}$ is smooth in space-time. To reach a contradiction, it is hence necessary to prove a strong fact about the limit, namely that the mild bounded ancient solution is zero. This is the purpose of the second step. For the case of the half-space, this step is nontrivial and involves a delicate analysis of the vorticity equation and its boundary condition.

The purpose of our work is to pave the way for a new method to prove regularity under the vorticity alignment and Type I condition. This new approach allows more flexibility in the rescaling procedure.
  To reach a contradiction, we do not need to show that the blow-up profile $u^{(\infty)}$ is zero.  Instead it suffices to prove that it is bounded at a specific point, which is much easier. The vorticity alignment condition serves the purpose of showing that $u^{(\infty)}$ is close to a two-dimensional flow in certain weak topologies.  This is sufficient to infer that $u^{(\infty)}$ is bounded at the desired point.

Our strategy  allows to prove regularity under the Type I condition, assuming vorticity alignment on concentrating sets, as in Theorem \ref{CAnconcentratingglobal}. 
 It also makes it possible to achieve local versions of the result of Giga, Hsu and Maekawa \cite[Theorem 1.3]{giga2014liouville}, and of our main result, Theorem \ref{CAnconcentratingglobal}. For such statements, we refer to Theorem \ref{theo.locgeomcrit} and Remark \ref{rem.vortalconsets} in Section \ref{sec.local}. As far as we know, the global or local statements with vorticity alignment on concentrating sets are new even for the case of the whole-space $\mathbb{R}^3$.

The keystone in our scheme is the stability of singularities for the Navier-Stokes equations. Stability of singularities and compactness arguments were also used in \cite[Theorem 4.1 and Remark 4.2]{AlbrittonBarkerBesov2018} and \cite[Theorem 3.1 and Lemma 3.3]{AlbrittonBarker2018local} to prove some regularity criteria. 
 The following lemma is a crucial tool. It is taken from the first author's Thesis \cite[Proposition 5.5]{barker2017uniqueness}; see also \cite[Proposition A.5]{AlbrittonBarker2018local} and \cite[Proposition 5.21]{pham2018} for subsequent genera\-lisations. The stability of singular points was first proven in the interior case by Rusin and \v{S}ver\'{a}k \cite{rusin2011minimal}. 
In the Lemma below and throughout this paper, we define $B^{+}(r):= B(r)\cap\mathbb{R}^3_{+}$ and $Q^{+}(1):= B^{+}(1)\times (-1,0)$. We refer the reader to subsection \ref{notationdef}, for the definition of suitable weak solutions in $Q^{+}(1)$. 
\begin{lemma}\label{stabilitysingularpointshalfspace}
Suppose $(u^{(k)},p^{(k)})$ 
 are suitable weak solutions to the Navier-Stokes equations in $Q^{+}(1)$.
Suppose that there exists a finite $M>0$ such that
\begin{equation}\label{seqbdd}
\sup_{k}\Big(\|u^{(k)}\|_{L_2^tL_{\infty}^x(Q^{+}(1))}+\|\nabla u^{(k)}\|_{L_{2}(Q^{+}(1))}+\|p^{(k)}\|_{L_{\frac{3}{2}}(Q^{+}(1))}\Big)=M<\infty.
\end{equation}
Furthermore assume that 
\begin{align}
\label{strongconverguk}
\lim_{k\rightarrow\infty}\|u^{(k)}-u^{(\infty)}\|_{L_{3}(Q^{+}(1))}=0,\\
\label{weakconvergpressure}
p^{(k)}\rightharpoonup p^{(\infty)}\,\,\textrm{in}\,\,L_{\frac 3 2}(Q^{+}(1)),\\
(0,0)\quad\mbox{is a singular point of}\quad u^{(k)}\quad \mbox{for all}\quad k=1,2\ldots\notag
\end{align} 
Then the above assumptions imply that $(u^{(\infty)},p^{(\infty)})$ is a suitable weak solution to the Navier-Stokes equations in $Q^{+}(1)$ with $(0,0)$ being a singular point of $u^{(\infty)}$.
\end{lemma}

To show more precisely how our proof of Theorem \ref{CAnconcentratingglobal} works, let us assume for contradiction that $u$ is singular at the space-time point $(0,T)$. Let $R_{k}\downarrow 0$ be \emph{any} sequence. We rescale in the following way 
\begin{equation}\label{e.rescalinguk}
u^{(k)}(y,s):= R_{k} u(R_{k}y, T+ R_{k}^2 s)\,\,\,\,\textrm{and}\,\, p^{(k)}(y,s):=R_{k}^{2} p(R_{k}y, T+ R_{k}^2 s), 
\end{equation}
for all $y\in\R^3_+$ and $s\in (-T/R_k^{2},0)$.
   We show that the Type I condition \eqref{TypeIhalf} implies the boundedness of the following scale-invariant quantities
\begin{multline}\label{apriori1}
\sup_{0<r<1}\Bigg\{\frac{1}{r} \sup_{T-r^2<t<T}\int\limits_{B^{+}(r)} | u(x,t)|^{2} dx+\frac{1}{r} \int\limits_{T-r^2}^T\int\limits_{B^+(r)} |\nabla u|^{2} dxdt\\+\frac{1}{r^{2}}\int\limits_{T-r^2}^T\int\limits_{B^+(r)} |p-(p)_{B^{+}(r)}|^{\frac32} dxdt\Bigg\}<\infty.
\end{multline} 
Hence \eqref{seqbdd} holds for $(u^{(k)},p^{(k)})$ and we can apply Lemma \ref{stabilitysingularpointshalfspace}. While this implication is well-known in the whole-space, it is a main technical block in the half-space. One of our main achievements in this paper is to overcome this technical block. We explain this in more details in the paragraph \ref{sec.scaleinv} below.

Applying Lemma \ref{stabilitysingularpointshalfspace} gives us the following. Namely the blow-up profile $u^{(\infty)}$ (obtained from $u^{(k)}$ in the limit) is an ancient mild solution,\footnote{In particular, it satisfies the Duhamel formulation on any compact subinterval of $(-\infty,0)$. We will sometimes refer to this property as being a `mild solution in $\mathbb{R}^3_{+}\times (-\infty,0)$'.} which has a singularity at $(0,0)$, satisfies the Type I assumption and has bounded scaled energy. Moreover, the continuous alignment condition for the vorticity implies that the vorticity direction of $u^{(\infty)}$ is constant in a large ball.  We can then reach a contradiction by the following proposition for $\bar u=u^{(\infty)}$, proved in Section \ref{sec.four}.

\begin{proposition}\label{2Dreductionbycompactness}
Suppose that $(\bar{u}, \bar{p})$ is a mild solution
 to the Navier-Stokes equations on $\mathbb{R}^3_{+}\times (-\infty,0)$ with $\bar{u}|_{\partial \mathbb{R}^3_{+}}=0.$
Let $\bar{\omega}:= \nabla\times \bar{u}$ and suppose $(\bar{u}, \bar{p})$ is a suitable weak solution to the Navier-Stokes equations on $\mathbb{R}^3_{+}\times (-\infty,0)$.
Furthermore suppose that
\begin{equation}\label{TypeI}
|\bar{u}(x,t)|\leq \frac{M}{\sqrt{-t}}
\end{equation}
and
\begin{multline}\label{Morreybound}
\sup_{0<r<\infty}\Bigg\{\frac{1}{r} \sup_{-r^2<t<0}\int\limits_{B^{+}(r)} | \bar{u}(x,t)|^{2} dx+\frac{1}{r} \int\limits_{Q^+(r)} |\nabla \bar{u}|^{2} dxdt\\+\frac{1}{r^{2}}\int\limits_{Q^+(r)} |\bar{p}-(\bar{p}(\cdot,t))_{B^{+}(r)}|^{\frac32} dxdt\Bigg\}\leq M'.
\end{multline}
Under the above assumptions the following holds true. For all $M$ and $M'>0$, there exists $\gamma(M,M')>0$ such that if 
\begin{equation}\label{almost2D}
\bar{\omega}(x,-t_{0})\cdot \vec{e}_{i}=0\,\,\,\,\,\textrm{in}\,\,\,\,B^{+}(\gamma(M,M')\sqrt{t_{0}})\,\, \mbox{for}\footnote{Arguments from \cite{barker2015ancient} demonstrate that $\bar{u}$ is $C^{\infty}$ in space-time on $\overline{\mathbb{R}^{3}_{+}}\times (-\infty,0)$, hence this pointwise condition on the vorticity is well defined.}\ i=2,3\,\,\,\,\mbox{and}\,\,\,\mbox{some}\,\,\,\,t_{0}\in(0,\infty)\footnote{Here, $\vec{e}_{2}=(0,1,0)$ and $\vec{e}_{3}=(0,0,1).$}
\end{equation}
then $(x,t)=(0,0)$ is a regular point for $\bar{u}$.
\end{proposition}

The main flexibility of our method lies in the fact that we can use \emph{any} sequence $R^{(k)}\downarrow 0$ in the rescaling procedure. Therefore, we can tune the sequence to our needs. In the case of the whole-space, we can take advantage of this versatility on time slices of the cone (Theorem \ref{CAnconcentratingglobalws} in Section \ref{improveremark}) and not the whole-cone as in Theorem \ref{CAnconcentratingglobal}.

\begin{remark}[alternative argument suggested by Seregin]\label{Seregincomment}
Gregory Seregin suggested to the authors that bounded scaled energy quantities for $u$ such as \eqref{apriori1} could be used to reprove the theorem of Giga, Hsu and Maekawa \cite[Theorem 1.3]{giga2014liouville} in the following way. His idea is to show that the existence of a singular solution  satisfying \eqref{vorticitydirection} and \eqref{apriori1}\footnote{Technically, \eqref{TypeIdef} must be assumed instead of \eqref{apriori1}. For the purpose of this discussion we overlook this point.} implies the existence of either a non-trivial two-dimensional whole-space or half-space mild bounded ancient solution $\bar u$ in $\Omega \times (-\infty,0)$ ($\Omega=\mathbb{R}^3$ or $\mathbb{R}^3_{+}$) satisfying 
\begin{equation}\label{Morreybd}\sup_{0<r<\infty} \sup_{-r^2<t<0}\frac{1}{r}\int\limits_{B(r)\cap \Omega} | \bar{u}(x_{1},x_{3},t)|^{2} dx_{1}dx_{2} dx_{3}<\infty.
\end{equation}
This then implies that either $\bar u\in L^{\infty}(-\infty,0; L^{2}(\mathbb{R}^2))$ or $L^{\infty}(-\infty,0; L^{2}(\mathbb{R}^2_{+}))$. Such a $\bar{u}$ then must be zero by applying a Liouville theorem proven in \cite{koch2009liouville} (for the whole-space case) and  a Liouville theorem by Seregin in \cite{seregin2015liouville} (for the half-space case). This contradiction then implies regularity under the assumptions \eqref{vorticitydirection} and \eqref{apriori1}.

Our approach avoids resorting to a Liouville theorem. Under the `ODE blow-up rate' assumption \eqref{TypeIhalf}, we gain additional information about the blow-up profile. Namely, $$|u^{(\infty)}(x,t)|\leq \frac{M}{\sqrt{-t}}.$$
This plays an important role in our proof. Specifically,  it ensures that for $t_{0}>0$, $u^{(\infty)}$ is the unique strong solution to the Navier-Stokes equations on $\mathbb{R}^3_{+}\times (-t_{0},0)$ with initial data $u^{(\infty)}(\cdot,-t_{0})$. It is not known if such considerations apply to the case when $u^{(\infty)}$ satisfies \eqref{Morreybd}. 
\end{remark}

We now explain how we get over the main technical block in this argument, namely to prove that the Type I assumption \eqref{TypeIhalf} implies the boundedness of scaled energies \eqref{apriori1}.

\subsection{Scale-invariant estimates}
\label{sec.scaleinv}

Let $\Omega$ be a domain in $\R^3$. We define the following scale-invariant quantities, which will be used throughout our work: for $x\in\bar\Omega$, $r\in(0,\infty)$
\begin{align}
\label{Adefx}
A(u,r;x,t):=\ &\sup_{t-r^2<s<t}\frac{1}{r}\int\limits_{B(x,r)\cap\Omega} | u(y,s)|^{2} dy,\\
\label{Edefx}
E(u,r;x,t):=\ &\frac{1}{r} \int\limits_{t-r^2}^t\int\limits_{B(x,r)\cap\Omega} |\nabla u|^{2} dyds,\\
\label{scaledpressurelambda32x}
D_{\frac32}(p,r;x,t):=\ &\frac{1}{r^{2}}\int\limits_{t-r^2}^t\int\limits_{B(x,r)\cap\Omega} |p-(p(\cdot,s))_{B(x,r)\cap\Omega}|^{\frac32} dyds.
\end{align}
Here and throughout this paper, $$(f)_{\Omega}:= \frac{1}{ |\Omega|} \int\limits_{\Omega} f(y) dy.$$ 
Below, we will often take $(x,t)=(0,0)$, as well as $\Omega=B^+(1)$, $B(1)$, $\R^3_+$, or $\mathbb{R}^3$. 
In this case, we have the following lighter notation:
\begin{equation}
\label{Adef}
A(u,r):=A(u,r;0,0),\quad E(u,r):=E(u,r;0,0),\quad D_{\frac32}(p,r;x,t):=D_{\frac32}(p,r;0,0).
\end{equation}
Away from boundaries it is known that for a suitable weak solution\footnote{We use the definition of `suitable weak solution' in $Q(1)$ given in \cite[Definition 1]{Seregin2018}. For the definition of suitable weak solutions in $Q^{+}(1)$, we refer the reader to subsection \ref{notationdef}.} $(u,p)$ of the Navier-Stokes equations in $Q(1)$, the Type I condition 
\begin{equation}\label{e.typeIR3}
|u(x,t)|\leq \frac{M}{\sqrt{-t}}\,\,\,\,\,\textrm{in}\,\,\,\,Q(1)
\end{equation}
implies
$$
\sup_{0<r<1}\big\{A(u,r)+E(u,r)+D_{\frac32}(p,r)\big\}<\infty.
$$
This was established by Seregin and Zajaczkowski in \cite{SerZajac}; see also Seregin and {\v{S}}ver{\'a}k \cite{Seregin2018} (statement and proof) and \cite[Lemma 2.5]{albritton2018local} (statement only). We show that this implication is also true up to a flat boundary. The following theorem is one of our main contributions in this paper, and the key technical tool to prove Theorem \ref{CAnconcentratingglobal} as well as its local versions.

\begin{theorem}\label{TypeIimpliesscaled}
Suppose that $(u,p)$ is a suitable weak solution of the Navier-Stokes equations in $Q^{+}(1)$ and satisfies the Type I bound 
\begin{equation}\label{e.typeIR3+loc}
|u(x,t)|\leq \frac{M}{\sqrt{-t}}\,\,\,\,\,\textrm{in}\,\,\,\,Q^+(1)
\end{equation}
Then one infers that
\begin{equation}\label{boundedscaledquantities}
\sup_{0<r<1}\big\{E(u,r)+A(u,r)+ D_{\frac{3}{2}}(p,r)\big\}\leq M'\big(M, A(u,1), E(u,1), D_{\frac{3}{2}}(p,1)\big).
\end{equation}
\end{theorem}

Let us now discuss a generalized notion of `Type I' blow-up. Suppose that for $\Omega=\mathbb{R}^3$ or $\mathbb{R}^3_{+}$ there is a suitable weak solution $u$ that loses smoothness at $T$. We then call\footnote{This definition was given in \cite{albritton2018local} for local solutions defined in a ball. See also \cite{seregin2019type} for a related definition for local solutions with no-slip on the flat part of the boundary.}  the singular point $(x_{0},T)\in \bar{\Omega}\times \{T\}$ a Type I singularity if there exists $R>0$ such that
\begin{align}
&\sup_{\substack{(B(x,r)\cap\Omega)\times (t-r^2,t)\\
\subset (B(x_{0},R)\cap\Omega) \times (T-R^2,T)}}\big(A(u,r;x,t)+E(u,r;x,t)+D_\frac32(p,r;x,t)\big)
<\infty.
\label{TypeIdef}
\end{align}
For the case of a suitable solution $u$ defined on the whole-space, it is discussed in \cite{albritton2018local} that the definition of Type I singularities given by \eqref{TypeIdef} is very natural and includes most other popular notions of `Type I' used in the literature. For the case of a suitable solution $u:\mathbb{R}^3_{+}\times (0,\infty)\rightarrow \mathbb{R}^3$ in the half-space with no-slip boundary condition, it was shown in \cite{Mik09} and \cite{seregin2019type} that the definition of Type I singularity \eqref{TypeIdef} is implied by 
 $$|u(x,t)|\leq \frac{M}{|x|}.$$
 However, it was previously not understood if, for solutions in the \emph{half-space} with no-slip boundary conditions, \eqref{TypeIdef} was consistent with the `ODE blow-up rate' notion of Type I singularities \eqref{e.typeIR3} and \eqref{e.typeIR3+loc} commonly found in the literature (see for instance \cite{GK85}). 
 Theorem \ref{TypeIimpliesscaled} fills this gap and demonstrates that, after all, \eqref{TypeIdef} is a reasonable notion of Type I singularity for suitable solutions in the half-space with no-slip boundary condition. 

The importance of Theorem \ref{TypeIimpliesscaled} lies in the fact that it links the natural notion of Type I blow-up \eqref{e.typeIR3+loc} to a workable notion of Type I. Indeed it turns out that the scale-invariant conditions \eqref{boundedscaledquantities} or \eqref{TypeIdef} are exactly what is needed to prove a number of results. This can be seen in particular in three situations:
\begin{enumerate}
\item The scale-invariant bound \eqref{boundedscaledquantities} ensures that the rescaled solutions $(u^{(k)},p^{(k)})$ defined by \eqref{e.rescalinguk} satisfy the uniform bound \eqref{seqbdd} required for the persistence of singularities.
\item The type I bound \eqref{TypeIdef} ensures that a mild bounded ancient solution originating from $(u,p)$ has some sort of decay at space infinity. This point is discussed in more details in Section \ref{sec.mbas} below.
\item In the paper \cite{BP18} on concentration phenomena for Type I blow-up solutions of the Navier-Stokes equations, a scale-critical condition such as \eqref{TypeIdef} was used to provide a control in $L_{2,uloc}(\R^3)$ of a rescaled solution.
\end{enumerate}

In order to prove Theorem \ref{TypeIimpliesscaled} a number of innovations are needed. 
Indeed, the case near a flat boundary is considerably more intricate than the interior case.
 The proof relies on the local energy inequality 
\eqref{localenergyequalityboundary}. The main difficulty lies in the estimate for the pressure term
\begin{equation}\label{e.estpressureQ(r)}
I(u,p,r,\Omega):=\int\limits_{-r^2}^0\int\limits_{B(r)\cap\Omega}pu\cdot\nabla\varphi dxds
\end{equation}
where $\varphi \in C_{0}^{\infty}(B(r)\times (-r^2,\infty))$ is a positive test function and $\Omega=B^{+}(1)$, see \eqref{localenergyequalityboundary}.

To highlight the difficulties concerned with the flat boundary, let us first discuss the proof of the simpler interior case of Theorem \ref{TypeIimpliesscaled} \cite{SerZajac,Seregin2018}.  
For the Navier-Stokes equations in the whole space (with sufficient decay),  the Calder\'{o}n Zygmund theory gives that the pressure $p$ can be estimated directly in terms of $u\otimes u$, which we write in a formal way $p\simeq u\otimes u$. 
 For the interior case of a suitable weak solution in $Q(1)$, this fact implies the decay estimate
\begin{equation}\label{e.estpressurewholespace}
D_{\frac{3}{2}}(p,\tau r)\leq C\big(\tau D_{\frac{3}{2}}(p,r)+\tau^{-2}C(u,r)\big),
\end{equation}
for a given $\tau\in (0,1)$ and all $r\in(0,1]$. Here, 
\begin{equation}\label{e.Cr}
C(u,r):=\frac1{r^2}\int_{Q(r)}|u|^3dxds
\end{equation}
is a scale-invariant quantity. To prove the interior version of Theorem \ref{TypeIimpliesscaled}, it then suffices to bound the right hand side of \eqref{e.Cr}. The game to play is to combine energy type quantities and the Type I bound \eqref{e.typeIR3} to obtain
\begin{align}
\label{e.estC(r)}
C(u,r)\lesssim_M\ & A^\mu(u,r)\Biggl(\frac1{r^3}\int\limits_{Q(r)}|u|^2dxds+\frac1{r}\int\limits_{Q(r)}|\nabla u|^2dxds\Biggr)^\theta \\
&\Downarrow\notag\\
C(u,r)+D_{\frac{3}{2}}(p,\tau r)\lesssim_M\ &\ep\Biggl(\frac1{r^3}\int\limits_{Q(r)}|u|^2dxds+\frac1{r}\int\limits_{Q(r)}|\nabla u|^2dxds\Biggr)\label{e.estCD(r)}\\
&\qquad\qquad+c(M,\ep, \kappa,\tau)+C\tau D_{\frac{3}{2}}(p,r)\notag
\end{align}
with 
\begin{equation}\label{e.defkappa}
0<\mu,\, \theta,\qquad \kappa:=\mu+\theta<1,\qquad \ep\in(0,1)\quad\mbox{and}\quad c(M,\ep,\kappa,\tau)>0. 
\end{equation}
This estimate is based on interpolation. 
Figure \ref{fig.interpol} shows how to obtain it in a clear way.  
Estimate \eqref{e.estC(r)} is critical to get the decay estimate for the energy for $\ep$ and $\tau$ small enough; see \cite[estimate (54)]{Seregin2018}. 

\begin{figure}[t]\label{fig.interpol}
\includegraphics[width=9.45cm,height=8cm]{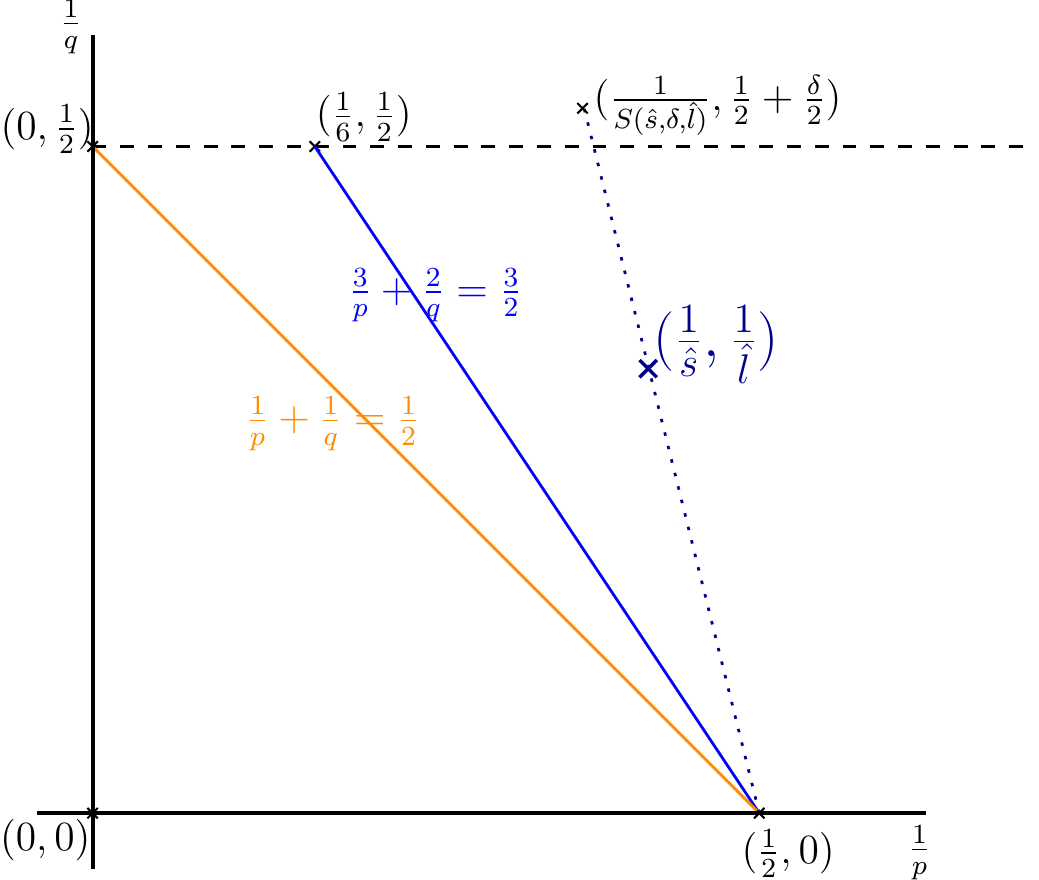}
\caption{Interpolation for $L_{\hat l}^tL_{\hat s}^x(Q^+(1))$}
\end{figure}

In the case of the half-space, on the contrary, we cannot estimate $p$ in terms of $u\otimes u$; in short $p\not\simeq u\otimes u$. This issue is brought up in particular in the work of Koch and Solonnikov \cite[Theorem 1.3]{KochSolonnikov}. Considering the unsteady Stokes system in the half-space with no-slip boundary conditions, they prove that there are divergence form source terms $\nabla\cdot F$, with $F\in L_q^{t,x}$, for which the pressure is not integrable in time. A possible alternative is to use maximal regularity for the Stokes system \cite{GS91}, i.e. $\nabla p\simeq u\cdot\nabla u$.  Such an estimate was  used for the local boundary regularity theory of the Navier-Stokes theory by Seregin in \cite{Ser02}. It remains a cornerstone for estimating the pressure locally near the boundary. Formally applying the maximal regularity estimate, together with the Poincar\'e-Sobolev and H\"{o}lder inequalities, the estimate of \eqref{e.estpressureQ(r)} (with $r=1$) turns into\begin{equation*}
\|u\|_{L^{t}_{l}L^{x}_{s}(Q^{+}(1))}\|p-(p)_{B^{+}(1)}(t)\|_{L^{t}_{l'}L^{x}_{s'}(Q^{+}(1))}\leq \|u\|_{L^{t}_{l}L^{x}_{s}(Q^{+}(1))}\|u\|_{L^{t}_{\hat{l}}L^{x}_{\hat{s}}(Q^{+}(1))}\|\nabla u\|_{L_{2}(Q^{+}(1))}
\end{equation*}
with 
\begin{equation}\label{holderexponent}
\frac{1}{l}+\frac{1}{\hat{l}}=\frac{1}{2}\,\,\,\,\textrm{and}\,\,\,\,\,\frac{1}{s}+\frac{1}{\hat{s}}=\frac{5}{6}.
\end{equation}
The presence of the gradient of $u$ makes this quantity `too supercritical'. Using \eqref{TypeIhalf} and sharp interpolation arguments (see Lemma \ref{interpolationwithTypeI}) give
\begin{equation}\label{optexponentinterpolate}
\|u\|_{L^{t}_{l}L^{x}_{s}(Q^{+}(1))}\|u\|_{L^{t}_{\hat{l}}L^{x}_{\hat{s}}(Q^{+}(1))}\lesssim_M (A(u,1))^{\lambda}
\end{equation}
with any $\lambda$ \textit{slightly} greater than $1-(\frac{1}{l}+\frac{1}{\hat{l}})=\frac{1}{2}$.
 Thus, any attempt to get an estimate similar to \eqref{e.estCD(r)} (for $\Omega=B^{+}(1)$)  by using maximal regularity estimates for the pressure and the Type I condition \eqref{TypeIhalf} will not work. In particular such a strategy will always produce a total power of energy quantities $\kappa\geq1$, see \eqref{e.defkappa}. 
The analysis of Figure \ref{fig.interpol} enables to understand the limitation on the exponent $\kappa$. Interpolating the $L_4^{t,x}$ norm of $u$ between the energy, i.e. the point $L_\infty^tL_2^x$, and the Type I condition, i.e. the region $L_q^tL_p^x$, $q<2$, yields a power of the energy too large. 

We overcome this difficulty in the half-space thanks to a new estimate on the pressure. The idea is to have an intermediate estimate between $p\simeq u\otimes u$, which is known to be false (see above the comment about \cite{KochSolonnikov}), and the maximal regularity estimate $\nabla p\simeq u\cdot\nabla u$, which is too supercritical. Such an estimate was discovered by Chang and Kang \cite[Theorem 1.2]{CK18}. Such `fractional' pressure estimate can be formally written as $\nabla^\beta p\simeq \nabla^\beta(u\otimes u)$, for $1\geq \beta$ sufficiently close to $1$. We give a new self-contained and straightforward proof of a similar estimate, see Proposition \ref{prop.changKang}. This new proof is based on the pressure formulas for the half-space established in \cite{MMP17a,MMP17b}. With the help of the `fractional' pressure estimate, we are able to implement the strategy of \cite{SerZajac,Seregin2018}. We believe that this is the first instance where the estimate of Chang and Kang \cite{CK18} (see also Proposition \ref{prop.changKang} below) is used and is really pivotal.

\subsection{Type I singularities, Morrey bounds and ancient solutions}
\label{sec.mbas}

{Let us now explain the independent interest of Theorem \ref{TypeIimpliesscaled} in clarifying the relationship between Type I singularities (see \eqref{TypeIdef} for a definition) 
of the Navier-Stokes equations in $\mathbb{R}^3_{+}$ with no-slip boundary condition and `whole-space mild bounded ancient solutions'  
or `half-space mild bounded ancient solutions'. These solutions form a special subclass of bounded solutions to Navier-Stokes equations on $\mathbb{R}^3\times (-\infty,0)$ and $\mathbb{R}^3_{+}\times (-\infty,0)$. For a definition, we refer to \cite{koch2009liouville} for the whole-space and \cite{seregin2014rescalings} and \cite{barker2015ancient} for the half-space.}

Suppose that $u: \mathbb{R}^3\times (0,\infty)\rightarrow \mathbb{R}^3$ is a solution that first loses smoothness at $T>0$. Let $t^{(k)}\uparrow T$, $y^{(k)}$ and $R_{k}\downarrow 0$ be selected as in \eqref{e.mbasrescaling}. 
In \cite{koch2009liouville}, it was shown that in such circumstances, a `zooming-in' procedure of the form
$$u^{(k)}(y,t)=R^{(k)}u\big(R^{(k)}(y+y^{(k)}), (R^{(k)})^2(t+t^{(k)})\big)$$
and
$$p^{(k)}(y,t)=(R^{(k)})^2 p\big(R^{(k)}(y+y^{(k)}), (R^{(k)})^2(t+t^{(k)})\big)$$
produces a non-zero whole-space mild bounded ancient solution; see also  \cite{sersve2009type} for the localised version of this statement. For the case that $u:\mathbb{R}^3_{+}\times (0,\infty)\rightarrow \mathbb{R}^3$ is a solution to the Navier-Stokes equations with no-slip boundary condition that first loses smoothness at $T>0$, the `zooming in' procedure produces either non-zero whole-space or non-zero half-space mild bounded ancient solutions; see \cite{seregin2014rescalings}, as well as \cite{AlbrittonBarker2018local} for the localised statement. In \cite{koch2009liouville} and \cite{seregin2014rescalings} the following Liouville conjectures were made
\begin{description}
\item[(C1)] Any whole-space mild bounded ancient solutions is a constant, see \cite{koch2009liouville}.\\
\item[(C2)] Any half-space mild bounded ancient solution is zero, see \cite{seregin2014rescalings}.
\end{description}
At the moment, these Liouville conjectures are only known to hold in special situations; see \cite{koch2009liouville}, \cite{sersve2009type}, \cite{albritton2018local} for special cases where \textbf{(C1)} holds and see \cite{seregin2015liouville} and \cite{giga2014liouville} for two-dimensional cases where \textbf{(C2)} holds.

When one assumes only `Type I' singularities at $T$ in the sense of \eqref{TypeIdef}, one gains extra information regar\-ding 
the whole-space (resp. half-space) mild bounded ancient solutions $u^{(\infty)}$ arising from $u$. In parti\-cular, one gets 
 \begin{equation}\label{globalMorreybound}
 \sup_{(x,t)\in \bar{\Omega}\times (-\infty,0], r>0}\big(A(u^{(\infty)},r;x,t)+E(u^{(\infty)},r;x,t)+D_\frac32(p^{(\infty)},r;x,t)\big)
  <\infty.
 \end{equation}
 Observe that if $u^{(\infty)}$ is a constant satisfying \eqref{globalMorreybound} then it must be zero. Consequently the validity of the Liouville conjectures \textbf{(C1)-(C2)} would imply that if one has a Type I singularity at $T$, the resulting mild bounded ancient solution must be zero. This would contradict with the fact that the mild bounded ancient solutions obtained from rescalings of a singular solution must be non-zero. 
In particular,
 \begin{itemize}
 \item[] \textbf{Validity of (C1) and (C2) rules out Type I singularities for finite energy solutions in $\mathbb{R}^3\times (0,\infty)$ or $\mathbb{R}^3_{+}\times (0,\infty)$ with no-slip boundary conditions.}
 \end{itemize}
 We \textit{strongly} emphasize that the main advantage of (and reason for) definition \eqref{TypeIdef} for Type I singularities is that condition \eqref{globalMorreybound} enables the reverse implication to be established.  Specifically, 
 \begin{itemize}
 \item[] \textbf{Failure of \textbf{(C1)} or \textbf{(C2)} for mild bounded ancient solutions satisfying \eqref{globalMorreybound} implies Type I singularity formation.}
 \end{itemize}
 This was shown in \cite{albritton2018local} for the whole-space case using the `persistence of singularities' established by Rusin and \v{S}ver\'{a}k in \cite{rusin2011minimal}. Subsequently, this was established for the half-space case by Seregin in \cite{seregin2019type}. As we mentioned before, bounds such as \eqref{globalMorreybound} are crucial for applying persistence of singularities arguments to rescalings of $(u^{(\infty)},p^{(\infty)})$.

\subsection{Outline of our paper}

In Section \ref{sec.ck} we revisit the Chang and Kang estimate. We prove Proposition \ref{prop.changKang}, which is our main tool to prove that the Type I condition \eqref{TypeIhalf} implies the boundedness of scaled energies \eqref{TypeIdef}. This is Theorem \ref{TypeIimpliesscaled}, which is proved in Section \ref{sec.scalecrit}. Section \ref{sec.four} is devoted to the proof of regularity criteria under vorticity alignment on concentrating sets, namely Theorem \ref{CAnconcentratingglobal} and the technical result Proposition \ref{2Dreductionbycompactness}. In Section \ref{improveremark} we discuss further improvements of our main result. The final part, Section \ref{sec.local}, handles localised versions of Theorem \ref{CAnconcentratingglobal}. In particular, we prove Theorem \ref{theo.locgeomcrit}. 
Here we have to deal with issues specific to the Navier-Stokes equations with a flat boundary in the local setting.

\subsection{A few notations and definitions}\label{notationdef}

For $\rho>0$, we define $B(\rho):= B(0,\rho)$ and $Q(\rho):= B(\rho)\times (-\rho^2,0)$; for the half-space, $B^{+}(\rho):= B(\rho)\cap\mathbb{R}^3_{+}$, $Q^{+}(\rho):= B^{+}(\rho)\times (-\rho^2,0)$ and $Q^+:=\R^3_+\times(-\infty,0)$. In addition, we define $\Gamma(0,\rho)=\{x\in B(0,\rho): x_{3}=0\}$. The notation $L_q^tL_p^x(Q(1))$ stands for $L_q(-1,0;L_p(B(1)))$, and similarly for other domains of $\R^3\times (-\infty,\infty)$ or $\R^3_+\times (-\infty,\infty)$. The notation for the space $W^{2,1}_{m,n}$ is standard and is taken for example from \cite{Ser06}. In \cite{Ser06}, several useful parabolic Sobolev embeddings are also stated. We also denote by $\mathring W^1_p(B(1))$ the closure $\overline{C^\infty_c(B(1))}^{\|\nabla\cdot\|_{L_p}}$.

The definition of mild solutions of the Navier-Stokes equations \eqref{noslip} in the half-space is given in \cite{barker2015ancient} in potential form and in \cite{MMP17a} using the abstract form of Stokes semigroup.

Now we give the definition of suitable weak solutions to the Navier-Stokes equations in $Q^{+}(r)$.\footnote{This definition is taken from \cite{Ser09} (Definition 1.3).}
\begin{definition}\label{suitableweaksolutionboundary} 
We say that the pair $u:Q^{+}(r)\rightarrow\mathbb{R}^3$ and $p:Q^{+}(r)\rightarrow\mathbb{R}$  is a suitable weak solution to the Navier-Stokes equations in $Q^{+}(r)$  if:
\begin{equation}\label{spacesboundary}
u\in L_{\infty}(-r^2,0; L_{2}(B^{+}(r))),\,\,\,\,\,
\nabla u\in L_{2}(Q^{+}(r)),\,\,\,\,\,\,\,\,
p\in L_{\frac{3}{2}}( Q^{+}(r)),
\end{equation}

\begin{equation}\label{weaksolutionrepeatboundary}
\int\limits_{Q^{+}(r)} (-u\cdot\partial_{t} \varphi- u\otimes u:\nabla\varphi+\nabla u:\nabla\varphi- p\nabla\cdot\varphi)dyds=0
\end{equation}
for any $\varphi\in C_{0}^{\infty}(Q^{+}(r))$ and
\begin{equation}\label{boundarycondition}
u(x,t)=0\,\,\,\textrm{for}\,\,x_3=0\,\, \textrm{and}\,\,\textrm{almost}\,\,\textrm{every}\,\, -r^2<t<0.
\end{equation}
 Additionally, it is required that the following local energy inequality holds for all positive $\varphi \in C_{0}^{\infty}(B^{+}(r)\times (-r^2,\infty))$ and almost every times $t\in (-r^2,0)$:
\begin{align}
\begin{split}\label{localenergyequalityboundary}
&\int\limits_{B^{+}(r)}\varphi|u(x,s)|^2dx+2\int\limits_{-r^2}^t\int\limits_{B^{+}(r)}\varphi |\nabla u|^2 dxds\\
\leq\ & \int\limits_{-r^2}^t\int\limits_{B^{+}(r)}|u|^2(\partial_t \varphi+\Delta\varphi)+u\cdot\nabla\varphi(|u|^2+2p)dxds.
\end{split}
\end{align}
\end{definition}
For a suitable weak solution $(u,p)$ in $Q^{+}(r)$, we have the following classification of points in $(B^+(0,r) \cup \Gamma(0,r)) \times (-r^2,0]$. In particular 
\begin{itemize}
\item[] $(y_0,s_0) \in  (B^+(0,r) \cup \Gamma(0,r)) \times (-r^2,0]$ is a `singular point' of $u$ $\Rightarrow $ 
\item[] $u\notin L_{\infty}((B^{+}(y_0,r) \times (s_0-r^2,s_0)) \cap (Q^{+}(r))$ for all sufficiently small $r>0.$
\end{itemize}
Moreover,  $(y_0,s_0)\in (B^+(0,r) \cup \Gamma(0,r)) \times (-r^2,0]$ is  defined to be a `regular point' of $u$ if it is not a singular point of $u$.

\section{Estimates for the pressure in the half-space: revisiting the Chang and Kang bound}
\label{sec.ck}

In this section, we focus on obtaining Chang and Kang-like estimates for the pressure in the half-space (see \cite{CK18}). We consider the Stokes system 
\begin{equation}\label{e.nse}
\partial_{t} u-\Delta u+\nabla p= \nabla\cdot F,\,\,\,\,\,\, \nabla\cdot u=0\,\,\,\mbox{in}\,\,\, \R^3_+\times (0,T)
\end{equation}
supplemented with the initial and boundary condition
\begin{equation}\label{e.bdaryinit}
u|_{\partial\R^3_+}(\cdot,t)=0,\,\,\,\,\,u(x,0)=0.
\end{equation}
Here $F=(F_{\alpha\beta})_{1\leq\alpha,\beta\leq 3}\in L_q(0,T;\mathring W^{1}_{p}(\R^3_+))$, for a fixed $T\in(0,\infty)$. A key point is that $F$ vanishes on $\partial\R^3_+$. Our main result is in the following proposition.

\begin{proposition}\label{prop.changKang}
For all $q\in[1,\infty]$, for all $p\in(1,\infty)$,  
for all $\kappa\in(\frac1p,1]$, for all $F\in L_q(0,T;\mathring W^{1}_{p}(\R^3_+))$, we have the following estimate for the pressure of \eqref{e.nse}-\eqref{e.bdaryinit}: $p=p^{Helm}_F+p^{harm}_F$, where
\begin{align*}
\|p^{Helm}_F\|_{L_q(0,T;L_p(\R^3_+))}\leq\ & C\|F\|_{L_q(0,T;L_p(\R^3_+))},\\
\|p^{harm}_F\|_{L_q(0,T;L_\infty^{x_3}(0,\infty;L_p^{x'}(\R^2)))}\leq\ & CT^{\frac1{2}(\kappa-\frac1p)}\big\|\|F\|_{L_p(\R^3_+)}^{1-\kappa}\|\nabla F\|_{L_p(\R^3_+)}^\kappa\big\|_{L_q(0,T)},
\end{align*}
with a constant $C(\kappa,p)\in(0,\infty)$.
\end{proposition}

This result is not completely new. A similar estimate is contained in the work of Chang and Kang \cite[Theorem 1.2, (1.14)]{CK18}. Notice that we gain boundedness in the vertical direction for $p^{harm}_F$, but we are not able to recover the $L_p^{x_3}$ integrability as in Chang and Kang. 
Our point here is to revisit the proof of the estimate of Chang and Kang. We give a proof based on elementary arguments, which takes advantage of new pressure formulas for the half-space discovered in \cite{MMP17a,MMP17b}. In particular, we avoid the use of space-time fractional Sobolev norms.

We decompose the pressure into
\begin{equation}\label{e.decpress}
p=p^{Helm}_{F}+p^{harm}_{F}.
\end{equation}
The Helmholtz pressure is given by
\begin{equation}\label{e.pHelm}
p^{Helm}_F(x',x_3,t)=-\int\limits_{\R^3_+}\nabla_zN(x',x_3,z',z_3)\cdot (\nabla\cdot F(z',z_3,t))dz'dz_3,
\end{equation}
for all $(x',x_3)\in \R^3_+$. Here $N$  
is the Neumann kernel for the Laplace equation in the half-space. We have the following formula: for all $x,\, z\in\R^3$, 
\begin{equation*}
N(x,z)=E(x-z)+E(x-z^*),\quad z^*=(z',-z_3),
\end{equation*}
where $E$ is the Green function for the Laplace equation in $\R^3_+$. 

For the harmonic part of the pressure, we rely on a decomposition of the pressure for the Stokes resolvent problem in the half-space. We follow the ideas in \cite{MMP17a,MMP17b} relying on an earlier decomposition of the pressure for the Stokes resolvent problem carried out in \cite{DHP01}. The harmonic part of the pressure is defined by
\begin{align}
&p^{harm}_{F}(x',x_3,t)\nonumber\\
=\ &\frac1{2\pi i}\int\limits_0^t\int\limits_{\Gamma}e^{\lambda(t-s)}\int\limits_{\R^3_+}q_\lambda(x'-z',x_3,z_3)\cdot(\mathbb P\nabla\cdot F(z',z_3,s))'dz'dz_3d\lambda ds\label{e.pharm}\\
=\ &\frac1{2\pi i}\int\limits_0^t\int\limits_{\Gamma}e^{\lambda(t-s)}\int\limits_{\R^3_+}\tilde q_\lambda(x'-z',x_3,z_3)(\mathbb P\nabla\cdot F(z',z_3,s))_3dz'dz_3d\lambda ds,\label{e.pharmtilde}
\end{align}
for all $(x',x_3)\in \R^3_+$. Here $\mathbb P$ denotes the Helmholtz-Leray projection and $\Gamma=\Gamma_\rho$ with $\rho\in (0,1)$ is the curve 
\begin{equation*}
\{\lambda\in \mathbb{C}~|~|{\rm arg}\, \lambda|=\eta, ~|\lambda|\geq \rho\} \cup \{\lambda\in \mathbb{C}~|~|{\rm arg}\, \lambda|\leq \eta, ~|\lambda|=\rho\}
\end{equation*} 
for some $\eta \in (\frac{\pi}{2},\pi)$. The kernel $q_\lambda$ (see \cite[Section 2]{MMP17a}) is defined by for all $x'\in\R^{2}$ and $x_3,\ z_3>0$,
\begin{align}\label{e.defqlambda}
q_\lambda(x',x_3,z_3):=i\int\limits_{\R^{2}}e^{ix'\cdot\xi}e^{-|\xi|x_3}e^{-\omega_\lambda(\xi)z_3}\left(\frac{\xi}{|\xi|}+\frac{\xi}{\omega_\lambda(\xi)}\right) d\xi\in\C^2.
\end{align}
where $\omega_\lambda(\xi)=\sqrt{\lambda+|\xi|^2}$, while the kernel $\tilde q_\lambda$ is defined by for all $x'\in\R^{2}$ and $x_3,\ z_3>0$,
\begin{align}\label{e.defqlambdatilde}
\tilde q_\lambda(x',x_3,z_3):=-\int\limits_{\R^{2}}e^{ix'\cdot\xi}e^{-|\xi|x_3}e^{-\omega_\lambda(\xi)z_3}\left(\frac{\omega_\lambda(\xi)}{|\xi|}+1\right) d\xi\in\C.
\end{align}
We will see below that the use of formula \eqref{e.pharmtilde} is more adapted to the study of $p^{harm}_F$. Indeed the vertical component of $\mathbb P\nabla\cdot F$ vanishes on the boundary of $\R^3_+$, which is not necessarily the case for the tangential component. Therefore, we focus now on formula \eqref{e.pharmtilde} and on the properties of $\tilde{q}_\lambda$. 

We need the following bounds for derivatives of $\tilde q_\lambda$: for all $\ep>0$, for all $\lambda\in S_{\pi-\ep}=\{\rho e^{i \theta}:\ \rho>0,\, \theta\in[-\pi+\ep,\pi-\ep]\}\subset\mathbb C$, there exists $C^*(\ep),\, c^*(\ep)\in(0,\infty)$ such that for all $x'\in\R^2$, $x_3,\, z_3>0$,
\begin{align}
\label{e.estnabla'q}
|\nabla'^{m}\tilde q_{\lambda}(x',x_3,z_3)|\leq\ &\frac{C^*e^{-c^*|\lambda|^\frac12z_3}}{(x_3+z_3+|x'|)^{1+m}}\left(|\lambda|^\frac12+\frac{1}{x_3+z_3+|x'|}\right),\quad m\in\{1,2\},\\
|\partial_{z_3}\tilde q_{\lambda}(x',x_3,z_3)|\leq\ &\frac{C^*e^{-c^*|\lambda|^\frac12z_3}}{(x_3+z_3+|x'|)}\left(|\lambda|+\frac{1}{(x_3+z_3+|x'|)^2}\right).
\label{e.estpartil3q}
\end{align}
These bounds are derived as the bounds for $q_\lambda$ in \cite[Proposition 3.7]{MMP17a}.

We focus now on the harmonic part of the pressure. We first compute the Helmholtz-Leray projection. From the work of Koch and Solonnikov \cite[Proposition 3.1]{KS02}, there exists {$G\in L^q(0,T;W^{1,p}(\R^3_+))$} such that
\begin{equation}\label{e.lerayprojG}
\mathbb P\nabla\cdot F=\nabla\cdot G.
\end{equation}
Moreover, we have the following formula for $G$: for all $1\leq\alpha,\, \beta\leq 3$,
\begin{align}\label{e.formulaG}
\begin{split}
G_{\alpha\beta}=\ &F_{\alpha\beta}-\delta_{\alpha\beta}F_{33}+(1-\delta_{3\alpha})\partial_{x_\beta}\Bigg(\int\limits_{\R^3_+}\partial_{z_\gamma}N(x,z)F_{\alpha\gamma}(z)dz\\
&+\int\limits_{\R^3_+}\big(\partial_{z_3}N(x,z)F_{3\alpha}(z)-\partial_{z_\alpha}N(x,z)F_{33}(z)\big)dz\Bigg).
\end{split}
\end{align}
This formula implies by the Calder\'on-Zygmund theory that for all $p\in(1,\infty)$, there exists a constant $C(p)\in(0,\infty)$ such that for almost every $t\in (0,T)$
\begin{equation}\label{e.estG}
\|G(\cdot,t)\|_{L_{p}(\R^3_+)}\leq C\|F(\cdot,t)\|_{L_{p}(\R^3_+)}\,\,\,\,\,\,\textrm{and}\,\,\,\,\,\|\nabla G(\cdot,t)\|_{L_{p}(\R^3_+)}\leq C\|\nabla F(\cdot,t)\|_{L_{p}(\R^3_+)}.
\end{equation}
In the study of \eqref{e.pharmtilde}, we have $\beta=3$. Notice that
\begin{equation}\label{e.propG33}
(\mathbb P\nabla\cdot F)_3=\sum_{\alpha=1}^2\partial_\alpha G_{\alpha 3}
\end{equation}
because $G_{33}=0$ by the formula \eqref{e.formulaG}. This cancellation property is essential since it enables to avoid the derivative of $\tilde q_\lambda$ in the vertical direction, which in view of \eqref{e.estpartil3q} has less decay in $x'$. 

Furthermore, the fact that the trace of $F(\cdot,t)$ vanishes on $\partial\R^3_+$ implies that the trace of $G_{\alpha3}(\cdot,t)$ also vanishes on $\partial\R^3_+$. Indeed, this property is clear for the first two terms in the right hand side of \eqref{e.formulaG}. Let us consider the third term. We define $K_{1}$ in the following way
\begin{equation*}
K_1(x)=\int\limits_{\R^3_+}\partial_{z_\gamma}N(x,z)F_{\alpha\gamma}(z)dz,\qquad\forall\ x\in\R^3_+.
\end{equation*}
We have by integration by parts using that the trace of $F(\cdot,t)$ vanishes on the boundary, for all $x\in\R^3_+$,
\begin{equation*}
K_1(x)=-\int\limits_{\R^3_+}N(x,z)\partial_{z_\gamma}F_{\alpha\gamma}(z)dz.
\end{equation*}
Hence, $K_1$ is a solution to the following Neumann problem
\begin{equation*}
-\Delta K_1=\nabla\cdot F^T\quad\mbox{in}\ \R^3_+,\qquad \partial_{x_3}K_1=F^T\cdot \vec{e}_3\equiv 0\quad\mbox{in}\ \R^2\times\{0\}.
\end{equation*}
For the last two terms in the right hand side of \eqref{e.formulaG}, we argue similarly. We let
\begin{equation*}
K_2(x)=\int\limits_{\R^3_+}\big(\partial_{z_3}N(x,z)F_{3\alpha}(z)-\partial_{z_\alpha}N(x,z)F_{33}(z)\big)dz.
\end{equation*}
Then calling $\tilde F_3$ the vector $\tilde F_3=-F_{33}\vec{e}_\alpha+F_{3\alpha}\vec{e}_3$, where $\vec{e}_\beta$, $\beta\in\{1,\ldots\, 3\}$ is the $\beta^{th}$ vector of the canonical basis of $\R^3$, we have
\begin{equation*}
K_2(x)=-\int\limits_{\R^3_+}N(x,z)\nabla\cdot \tilde F(z)dz.
\end{equation*}
We then conclude as for $K_1$ above.

For the proof of Proposition \ref{prop.changKang}, we rely on the following straightforward lemma.

\begin{lemma}\label{lem.vertint}
Let $p\in(1,\infty)$ and $\kappa\in(\frac1p,1]$. Let $h\in \mathring W^1_p(0,\infty)$. We have that for all $\rho\in(0,\infty)$, there exists a constant $C(p,\kappa)\in(0,\infty)$ such that
\begin{equation}\label{e.lem.vertint}
\Bigg|\int\limits_0^\infty\frac{e^{-\rho z_3}}{z_3}h(z_3)dz_3\Bigg|\leq C\rho^{-\kappa+\frac1p}\|h\|_{L_p(0,\infty)}^\frac{1-\kappa}p\|\partial_{z_3} h\|_{L_p(0,\infty)}^\frac\kappa p.
\end{equation}
\end{lemma}

\begin{proof}[Proof of Lemma \ref{lem.vertint}]
H\"older's inequality with the parameters $\beta,\, \beta',\, \gamma,\, \gamma'\in[1,\infty]$ such that
\begin{equation*}
\kappa\beta=p,\quad (1-\kappa)\beta'\gamma=p,\quad \frac1\beta+\frac1{\beta'}=1,\quad \frac1\gamma+\frac1{\gamma'}=1
\end{equation*}
implies
\begin{align*}
\int\limits_0^\infty\frac{e^{-\rho z_3}}{z_3}h(z_3)dz_3
\leq \Bigg(\int\limits_0^\infty\frac{e^{-\rho\beta'\gamma' z_3}}{z_3^{(1-\kappa)\beta'\gamma'}}dz_3\Bigg)^\frac1{\beta'\gamma'}\Bigg(\int\limits_0^\infty|h|^pdz_3\Bigg)^\frac1{\beta'\gamma}\Bigg(\int\limits_0^\infty\frac{|h|^p}{z_3^p}dz_3\Bigg)^\frac1{\beta}.
\end{align*}
This yields \eqref{e.lem.vertint} by using the one-dimensional Hardy inequality and the relations for the H\"older exponents. In particular, $\kappa>\frac{1}{p}$ and $\beta'\gamma'=\frac{p}{p-1}$ gives the convergence of the first integrand in the product.
\end{proof}

\begin{proof}[Proof of Proposition \ref{prop.changKang}]
We investigate separately the Helmholtz and the harmonic pressures. Let $q\in [1,\infty]$, $p\in(1,\infty)$ and $\kappa\in(\frac1p,1]$ be fixed.

\noindent{\bf Step 1: Helmholtz pressure}\\
From the formula \eqref{e.pHelm}, Calder\'on-Zygmund type estimates imply that
\begin{equation*}
\|p^{Helm}_F\|_{L_q(0,T;L_p(\R^3_+))}\leq C\|F\|_{L_q(0,T;L_p(\R^3_+))}
\end{equation*}
for a constant $C(p)\in(0,\infty)$.

\noindent{\bf Step 2: harmonic pressure}\\
Our starting point to estimate the harmonic part of the pressure is formula \eqref{e.pharmtilde}, which we combine with the formula \eqref{e.propG33} for the source term. We have for almost all $x_3\in(0,\infty),\, x'\in\R^2,\, s\in(0,T)$,
\begin{align*}
&\int\limits_0^\infty\int\limits_{\R^2}\tilde q_\lambda(x'-z',x_3,z_3)(\mathbb P\nabla\cdot F)_3(z',z_3,s)dz'dz_3\\
=\ &\sum_{\alpha=1}^2\int\limits_0^\infty\int\limits_{\R^2}\tilde q_\lambda(x'-z',x_3,z_3)\partial_{z_\alpha} G_{\alpha 3}(z',z_3,s)dz'dz_3\\
=\ &\sum_{\alpha=1}^2\int\limits_0^\infty\int\limits_{\R^2}\partial_{x_\alpha}\tilde q_\lambda(x'-z',x_3,z_3)G_{\alpha 3}(z',z_3,s)dz'dz_3\\
=\ &I(x',x_3,s).
\end{align*}
As we stressed above, we have only tangential derivatives falling on $\tilde q_\lambda$, so that we rely on \eqref{e.estnabla'q}. 
We first show that for all $\alpha\in\{1,2\}$, $x_3,\, z_3>0$, the kernel $k_{x_3,z_3}$ defined by for all $x'\in\R^2$
\begin{equation*}
k_{\alpha,x_3,z_3}(x')=\bigg(1+\frac1{|\lambda|^\frac12z_3}\bigg)^{-1}|\lambda|^{-\frac12}e^{c^*|\lambda|^\frac12z_3}\partial_{x_\alpha}\tilde q_\lambda(x',x_3,z_3)
\end{equation*}
is a Calder\'on-Zygmund kernel. We recall that $c^*$ is the constant appearing in \eqref{e.estpartil3q}. We have for all $x'\in\R^2$
\begin{align}
\label{e.boundka}
|k_{\alpha,x_3,z_3}(x')|\leq\ & \frac{C^*}{(x_3+z_3+|x'|)^2}\leq \frac{C^*}{|x'|^2},\\
|\nabla' k_{\alpha,x_3,z_3}(x')|\leq\ & \frac{C^*}{(x_3+z_3+|x'|)^3}\leq \frac{C^*}{|x'|^3},
\label{e.boundkb}
\end{align}
where $C^*$ is the constant in \eqref{e.estpartil3q}. 
A key point is that the bounds \eqref{e.boundka} and \eqref{e.boundkb} on $k_{x_3,z_3}$ are uniform in $x_3,\, z_3$ and $\lambda\in S_{\pi-\ep}$. Moreover, from \eqref{e.defqlambdatilde} we have the antisymmetry
\begin{equation}\label{e.anti}
k_{\alpha,x_3,z_3}(x')=-k_{\alpha,x_3,z_3}(-x'),\quad\forall x'\in\R^2.
\end{equation}
Hence, \eqref{e.boundka}, \eqref{e.boundkb} and \eqref{e.anti} imply that $k_{x_3,z_3}$ is a Calder\'on-Zygmund kernel. Therefore, for all $p\in(1,\infty)$, there exists $C(p)\in(0,\infty)$ such that for all  $h\in L^p(\R^2)$
\begin{equation}\label{e.bddkLp}
\|k_{\alpha,x_3,z_3}\ast h\|_{L_p(\R^2)}\leq C\|h\|_{L_p(\R^2)}.
\end{equation}
Notice that the constant in \eqref{e.bddkLp} is uniform in $x_3,\, z_3>0$ and $\lambda\in S_{\pi-\ep}$ because of the uniformity in \eqref{e.boundka} and \eqref{e.boundkb}. We now turn to estimating the harmonic pressure. We have for almost all $s\in(0,T)$,
\begin{align*}
&\|I(\cdot,s)\|_{L_\infty(0,\infty;L_p(\R^2))}\\
\leq\ &\sum_{\alpha=1}^2\left\|\int\limits_0^\infty\Bigg\|\int\limits_{\R^2}\partial_{x_\alpha}\tilde q_\lambda(x'-z',x_3,z_3)G_{\alpha 3}(z',z_3,s)dz'\Bigg\|_{L_p(\R^2)}dz_3\right\|_{L_\infty(0,\infty)}\\
\leq\ &C\sum_{\alpha=1}^2\Bigg\|\int\limits_0^\infty|\lambda|^\frac12e^{-c^*|\lambda|^\frac12z_3}\Big(1+\frac{1}{|\lambda|^\frac12z_3}\Big)\|G_{\alpha 3}(\cdot,z_3,s)\|_{L_p(\R^2)}dz_3\Bigg\|_{L_\infty(0,\infty)}\\
\leq\ &
C\sum_{\alpha=1}^2\Bigg\|\int\limits_0^\infty e^{-c|\lambda|^\frac12z_3}\frac{\|G_{\alpha 3}(\cdot,z_3,s)\|_{L_p(\R^2)}}{z_3}dz_3\Bigg\|_{L_\infty(0,\infty)},
\end{align*}
for some $c\in(0,c^*)$. Notice that we used the bound
$$|\lambda|^{\frac{1}{2}}e^{-c^*|\lambda|^{\frac{1}{2}}z_{3}}=\frac{1}{z_{3}}\times z_{3}|\lambda|^{\frac{1}{2}}e^{-c^*|\lambda|^{\frac{1}{2}}z_{3}}\leq \frac{Ce^{-\frac{c^*}2|\lambda|^{\frac{1}{2}}z_{3}}}{z_{3}}.$$
Applying now Lemma \ref{lem.vertint}, we find for almost every $s\in(0,T)$,
\begin{equation*}
\|I(\cdot,s)\|_{L_\infty(0,\infty;L_p(\R^2))}
\leq\ C|\lambda|^{-\frac12(\kappa-\frac1p)}\|G(\cdot,s)\|_{L_p(\R^3_+)}^{1-\kappa}\|\partial_{z_3}G(\cdot,s)\|_{L_p(\R^3_+)}^\kappa,
\end{equation*}
with a constant $C(\kappa,p)\in(0,\infty)$. Remark that in order to apply Lemma \ref{lem.vertint}, we used the crucial fact that $G_{\alpha 3}$ has zero trace on the boundary. It remains to compute the integral in $\lambda$ and to estimate the convolution in time. We have for almost every $s\in(0,T)$,
\begin{align*}
&\left|\int\limits_{\Gamma}e^{\lambda(t-s)}\|I(\cdot,s)\|_{L_\infty(0,\infty;L_p(\R^2))}d\lambda\right|\\
\leq\ & \int\limits_{\Gamma}e^{\Rel\lambda(t-s)}\|I(\cdot,s)\|_{L_\infty(0,\infty;L_p(\R^2))}|d\lambda|\\
\leq\ &C(p)\int\limits_0^\infty e^{(t-s)r\cos\eta}r^{-\frac12(\kappa-\frac1p)}dr\|G(\cdot,s)\|_{L_p(\R^3_+)}^{1-\kappa}\|\partial_{z_3}G(\cdot,s)\|_{L_p(\R^3_+)}^\kappa\\
\leq\ &C(\eta,\kappa,p)(t-s)^{\frac12(\kappa-\frac1p)-1}\|G(\cdot,s)\|_{L_p(\R^3_+)}^{1-\kappa}\|\partial_{z_3}G(\cdot,s)\|_{L_p(\R^3_+)}^\kappa.
\end{align*}
Since $\kappa>\frac1p$, the function $s^{\frac12(\kappa-\frac1p)-1}\mathbf 1_{(0,T)}(s)$ is integrable. Therefore by convolution in time
\begin{align*}
&\Bigg\|\int\limits_0^t\int\limits_{\Gamma}e^{\lambda(t-s)}\left\|I(\cdot,s)\right\|_{L_\infty(0,\infty;L_p(\R^3_+))}d\lambda ds\Bigg\|_{L_q(0,T)}\\
\leq\ & C(\eta,\kappa,p)T^{\frac12(\kappa-\frac1p)}\left\|\|G(\cdot,s)\|_{L_p(\R^3_+)}^{1-\kappa}\|\partial_{z_3}G(\cdot,s)\|_{L_p(\R^3_+)}^\kappa\right\|_{L_q(0,T)}.
\end{align*}
We conclude by using the bound \eqref{e.estG} on $G$.
\end{proof}

\section{Boundedness of scale-critical quantities}
\label{sec.scalecrit}

The goal of this section is to prove Theorem \ref{TypeIimpliesscaled} stated in the Introduction. The result follows from careful estimates of the terms in the right hand side of the local energy inequality \eqref{localenergyequalityboundary}. Our analysis involves a delicate study of the pressure term, see Section \ref{sec.pressscaleinv}.

In this section, all the parabolic cylinders are centered at $(0,0)$. We recall that $E(u,r)$, $A(u,r)$ and $D(u,r)$ are defined by \eqref{Adef}. We first estimate scale-critical quantities involving the velocity. These estimates play a crucial role in Section \ref{sec.pressscaleinv} when it comes to estimate the pressure.

\begin{lemma}\label{interpolationwithTypeI}
Suppose that $u\in C_{w}^{t}L_{2}^{x}(B^{+} \times (-1,0))$, $\nabla u\in L_{2}(B^{+}\times (-1,0))$ and the local ODE blow-up rate \eqref{e.typeIR3+loc}. Furthermore, assume 
\begin{equation}\label{indicecondition1}
\hat{s}, \hat{l}>2\,\,\,\,\,\,\,\textrm{and}\,\,\,\,\,\frac{1}{\hat{s}}+\frac{1}{\hat{l}}>\frac{1}{2}.
\end{equation}
Then the above assumptions imply that for 
\begin{equation}\label{deltahypothesis}
\delta\in \Big(0,\min\Big(\frac{\frac{1}{\hat{s}}+\frac{1}{\hat{l}}-\frac{1}{2}}{{\frac{1}{2}-\frac{1}{\hat{s}}}},1\Big)\Big)
\end{equation}
we obtain
\begin{equation}\label{interpolationbound}
\|u\|_{L^{t}_{\hat{l}}L^{x}_{\hat{s}}(Q^{+}(1))}\leq C(M,\delta,\hat{s},\hat{l})A(u,1)^{\frac12-\frac{1}{\hat{l}(1+\delta)}}.
\end{equation}
\end{lemma}
\begin{proof}
First we note that under the hypothesis on $\hat{s}$, $\hat{l}$ and $\delta$ we have
\begin{equation}\label{interpolationcondition1}
0< \frac{\Big(\frac{1}{\hat{s}}-\frac{1}{2}\Big)\Big(\frac{1}{2}+\frac{\delta}{2}\Big)}{\frac{1}{\hat{l}}}+ \frac{1}{2}<\frac{1}{\hat{s}}.
\end{equation}
Define $$\frac{1}{S(\hat{s},\delta, \hat{l})}:= \frac{\Big(\frac{1}{\hat{s}}-\frac{1}{2}\Big)\Big(\frac{1}{2}+\frac{\delta}{2}\Big)}{\frac{1}{\hat{l}}}+ \frac{1}{2}.$$
Observe that for $\theta= 1-\frac{2}{\hat{l}(1+\delta)}\in (0,1)$ we have
\begin{equation}\label{interpolationidentity}
\Big(\frac{1}{\hat{s}}, \frac{1}{\hat{l}}\Big)= \theta\Big(\frac{1}{2},0\Big)+(1-\theta)\Big( \frac{1}{S(\hat{s}, \delta, \hat{l})}, \frac{1}{2}+\frac{\delta}{2}\Big).
\end{equation}
Thus we may interpolate $L^{t}_{\hat{l}}L^{x}_{\hat{s}}$ between $L^{t}_{\infty}L^{x}_{2}$ and $L^{t}_{\frac{2}{1+\delta}}L^{x}_{S(\hat{s}, \delta, \hat{l})}$ (see Figure \ref{fig.interpol} on page \pageref{fig.interpol}) to get
\begin{equation}\label{interpolationcalculation}
\|u\|_{L^{t}_{\hat{l}}L^{x}_{\hat{s}}(Q^{+}(1))}\leq \|u\|_{L^{t}_{\frac{2}{1+\delta}}L^{x}_{S(\hat{s}, \delta, \hat{l})}(Q^{+}(1))}^{\frac{2}{\hat{l}(1+\delta)}}A(u,1)^{\frac12-\frac{1}{\hat{l}(1+\delta)}}.
\end{equation}
We then conclude by controlling $\|u\|_{L^{t}_{\frac{2}{1+\delta}}L^{x}_{S(\hat{s}, \delta, \hat{l})}(Q^{+}(1))}$ by \eqref{e.typeIR3+loc}.
\end{proof}
\begin{remark}\label{notation}
From now on when using  above lemma, we will write for convenience
$$ \|u\|_{L^{t}_{\hat{l}}L^{x}_{\hat{s}}(Q^{+}(1))}\leq A(u,1)^{(\frac12-\frac{1}{\hat{l}})_{+}}$$
with the understanding that $(\frac12-\frac{1}{\hat{l}})_{+}$ is greater than (but arbitrarily close to) $\frac12-\frac{1}{\hat{l}}$. 

\end{remark}
Before stating the next lemma, we introduce the notation
\begin{equation}\label{scaledvelocity}
H_{\lambda'}(u,r):=\frac{1}{r^{5-2\lambda'}}\int\limits_{Q^{+}(r)} |u|^{2\lambda'} dxdt
\end{equation}
\begin{equation}\label{scaledfractional}
F_{\lambda', \beta}(u,r):= \frac{1}{r^{5-2\lambda'-\lambda'\beta}}\int\limits_{-r^2}^{0} \|u\|_{L_{2\lambda'}(B^{+}(r))}^{2\lambda'(1-\beta)}\||u||\nabla u|\|_{L_{\lambda'}(B^{+}(r))}^{\beta\lambda'}dt.
\end{equation}
\begin{lemma}\label{fractionalsobolevwithTypeI}
Suppose that $u\in C_{w}^{t}L_{2}^{x}(B^{+} \times (-1,0))$, $\nabla u\in L_{2}(B^{+}\times (-1,0))$ 
and the local ODE blow-up rate \eqref{e.typeIR3+loc}. 
Define
\begin{equation}\label{lambda'def}
\frac{1}{\lambda'}:=\frac{3}{4}-\delta.
\end{equation}
Assume
\begin{equation}\label{deltacondition}
\delta\in \Big(0,\frac{1}{12}\Big)
\end{equation}
and
\begin{equation}\label{betacondition}
\frac{1}{\lambda'}<\beta< 1-4\delta
\end{equation}
Then under the above assumptions, there exists $\mu(\beta,\delta),\, \theta(\delta)\in (0,1)$ such that for all $r\in(0,1]$
\begin{equation}\label{lowerorderestimate}
H_{\lambda'}(u,r)\leq C(M,\delta,\beta,\lambda')A(u,r)^{\theta}
\end{equation}
\begin{equation}\label{fractionalsobolevspaceestimate}
F_{\lambda',\beta}(u,r)\leq C(M,\delta,\beta,\lambda')\big( A(u,r)+E(u,r)\big)^{\mu}.
\end{equation}
\end{lemma}
\begin{proof}
Using a scaling argument, we may assume without loss of generality that $r=1$. 
Since $2\lambda'\in (2,3)$ we can immediately apply the previous lemma to show that for some $\theta(\delta)\in (0,1)$ we have
$$H_{\lambda'}(u,1)\leq C(M,\delta) A(u,1)^{\theta}.$$
Let us now focus on the more complicated term $F_{\lambda',\beta}(u,1)$.
Using H\"{o}lder's inequality we get
\begin{equation}\label{IIest1}
F_{\lambda',\beta}(u,1)\leq \|u\|_{L^{t}_{D}L^{x}_{C}(Q^{+}(1))}^{\lambda'\beta} E(u,1)^{\frac{\lambda'\beta}2} \|u\|^{2(1-\beta)\lambda'}_{L^{t}_{2\hat{Q}}L^{x}_{2\lambda'}(Q^{+}(1))}.
\end{equation}
Here, we have
\begin{equation}\label{Holder1}
\frac{\beta}{\tilde{Q}}+\frac{1-\beta}{\hat{Q}}=\frac{1}{\lambda'},
\end{equation}
\begin{equation}\label{Holder2}
\frac{1}{C}+\frac{1}{2}=\frac{1}{\lambda'}
\end{equation}
and
\begin{equation}\label{Holder3}
\frac{1}{D}+\frac{1}{2}=\frac{1}{\tilde{Q}}.
\end{equation}
In order to apply the previous Lemma, we now show that we can select such indices with  $C,D\in (2,\infty)$, $\hat{Q}\in (1,\infty)$, $$ \frac{1}{C}+\frac{1}{D}>\frac{1}{2}$$  and $$\frac{1}{\lambda'}+\frac{1}{\hat{Q}}>1.$$
Using \eqref{lambda'def}-\eqref{deltacondition}, we see that this is equivalent to selecting $\hat{Q}$ and $\tilde{Q}$ such that \eqref{Holder1} is satisfied with $$\frac{1}{4}+\delta<\frac{1}{\hat{Q}}<1\,\,\,\,\,\,\,\,\textrm{and}\,\,\,\,\,\,\,\,\frac{3}{4}+\delta<\frac{1}{\tilde{Q}}<1. $$ 
By using \eqref{lambda'def}-\eqref{betacondition} together with the intermediate value theorem, we see that it is possible to select such indices. 
This then allows us to apply the previous lemma to \eqref{IIest1}, which gives
\begin{equation}\label{IIest2}
F_{\lambda',\beta}(u,1)\leq C(A(u,1)+E(u,1))^{\lambda'\big(\beta(\frac12-\frac{1}{D})_{+}+\frac\beta2+(1-\beta)(1-\frac{1}{\hat{Q}})_{+}\big)}.
\end{equation}
It is therefore sufficient to show
$$\lambda'\big(\beta(\tfrac12-\tfrac{1}{D})+\tfrac\beta2+(1-\beta)(1-\tfrac{1}{\hat{Q}})\big)<1. $$
This can be verified by using \eqref{betacondition} and \eqref{Holder1}-\eqref{Holder3}.
\end{proof}

\subsection{Scaled Pressure estimates under Type I assumption}
\label{sec.pressscaleinv}
For convenience, let us introduce the scaled quantities
\begin{equation}\label{scaledpressurelambda'}
D_{\lambda'}(p,r):=\frac{1}{r^{5-2\lambda'}}\int\limits_{Q^{+}(r)} |p-(p)_{B^{+}(r)}|^{\lambda'} dxdt
\end{equation}
In this subsection, we let $\lambda'$ be as in Lemma \ref{fractionalsobolevwithTypeI}. We remark that \eqref{lambda'def}-\eqref{deltacondition} imply that $\lambda'\in(1,\frac{3}{2})$, so such a quantity is well defined for suitable weak solutions. 
 We will need the proposition below taken from \cite{Ser09}.
\begin{proposition}\label{localboundaryregstokes}
Let $m,n$ and $s$ be such that $1<m<\infty$, $1<n\leq \infty$ and $m\leq s<\infty$. Suppose $\nabla u\in L_{n}^{t}L_{m}^{x}(Q^{+}(1))$, $p\in L_{n}^{t}L_{m}^{x}(Q^{+}(1))$ and $f\in L_{n}^{t}L_{s}^{x}(Q^{+}(1))$.
In addition, suppose that
$$
\partial_{t}u-\Delta u+\nabla p=f,\qquad\nabla\cdot u=0\qquad \mbox{in}\quad Q^{+}(1)
$$
and suppose $u$ satisfies the boundary condition 
$$u=0\,\,\,\,\,\,\textrm{on}\,\,\,\,\,\,{x_{3}=0}.$$
Then, we conclude that $\nabla p\in L_{n}^{t}L_{s}^{x}(Q^{+}(1/2))$. Furthermore, the estimate
\begin{multline*}
\|\nabla p\|_{L_{n}^{t}L_{s}^{x}(Q^{+}(1/2))} \\
\leq c(s,n,m)\big(\|f\|_{L_{n}^{t}L_{s}^{x}(Q^{+}(1))}+\| u\|_{L_{n}^{t}L_{m}^{x}(Q^{+}(1))}+\|\nabla  u\|_{L_{n}^{t}L_{m}^{x}(Q^{+}(1))}+\|p\|_{L_{n}^{t}L_{m}^{x}(Q^{+}(1))}\big)
\end{multline*}
holds.
\end{proposition}
\begin{remark}\label{constantsscaling}
We will use that result for $\lambda'=m=n$ and $f=0$, a scaling argument shows
\begin{equation}\label{harmscaled}
\|\nabla p\|_{L^{t}_{\lambda'}L^{x}_{s}(Q^{+}(\frac{r}{4}))}\leq \frac{r^{\frac{3}{s}+\frac{2}{\lambda'}}}{r^{\frac{5}{\lambda'}+1}}\Big(\frac{1}{r}\|u\|_{L_{\lambda'}(Q^{+}(\frac{r}{2}))}+\|\nabla u\|_{L_{\lambda'}(Q^{+}(\frac{r}{2}))}+\|p-(p)_{B^{+}(\frac{r}{2})}\|_{L_{\lambda'}(Q^{+}(\frac{r}{2}))}\Big)
\end{equation} 
\end{remark}
Let us now state the main estimate for the pressure when the Type I bound is satisfied. We state this as a lemma.
\begin{lemma}\label{pressuredecay}
Suppose that $(u,p)$ is a suitable weak solution to the Navier-Stokes equations on $Q^{+}(1)$. Furthermore, suppose $u$ satisfies the Type I bound \eqref{e.typeIR3+loc} and that $\lambda'$, $\beta$ and $\delta$ are as in Lemma \ref{fractionalsobolevwithTypeI}. Then there exist $\theta \in (0,1)$ and $\mu\in (0,1)$ such that for any $\lambda'<s<\infty$, $r\in (0,1)$ and $\tau\in (0,\frac{1}{4})$ we have
\begin{equation}\label{pressureestimatessmallscales}
D_{\lambda'}(p,\tau r)\leq C\tau^{(3-\frac{2}{\lambda'}-\frac{3}{s})\lambda'}(E(u,r)+A(u,r))^{\frac{\lambda'}2}+C\tau^{(3-\frac{2}{\lambda'}-\frac{3}{s})\lambda'} D_{\lambda'}(p,r)+$$$$+\frac{C}{\tau^{5-2\lambda'}}
\Big({(A(u,r))}^{\theta}+(E(u,r)+A(u,r))^{\mu}\Big).
\end{equation}
Here, the constant $C(M, \lambda',s,\beta,\delta)\in(0,\infty)$.
\end{lemma}
\begin{proof}

\noindent\textbf{Step 1: splitting using an initial boundary value problem}\\
Let $\varphi_{r}\in C_{0}^{\infty}(B^{+}(r))$ with
$\varphi_{r}=1$ on $B^{+}(\frac{r}{2})$, $0\leq \varphi_r\leq 1$ and $|\nabla \varphi_r|\leq \frac{c}{r}.$
Consider the initial boundary value problem
\begin{align*}
&\partial_{t} u^{1}-\Delta u^{1}+\nabla p^{1}= -\nabla\cdot (u\otimes u\varphi_{r})\,\,\,\,\,\,\mbox{in}\quad \mathbb{R}^3_{+}\times (-r^2,0),\\
&\nabla\cdot u^{1}=0\quad\mbox{and}\quad u^{1}(x',0,t)=0,\\
&u^{1}(\cdot,-r^2)=0.
\end{align*}
Using Proposition \ref{prop.changKang} with $F:=(u\otimes u)\varphi_{r}$, $p=q=\lambda'$ and $\kappa=\beta>\frac{1}{\lambda'}$, we have 
$p^{1}=p^{Helm}_F+p^{harm}_F$, with
\begin{align*}
\|p^{Helm}_F\|_{L_{\lambda'}(-r^2,0;L_{\lambda'}(\R^3_+))}\leq\ & C\|u\otimes u\varphi_{r}\|_{L_{\lambda'}(-r^2,0;L_{\lambda'}(\R^3_+))},\\
\|p^{harm}_F\|_{L_{\lambda'}(-r^2,0;L_\infty^{x_3}(0,\infty;L_{\lambda'}^{x'}(\R^2)))}\leq\ & Cr^{(\beta-\frac{1}{\lambda'})}\big\|\|(u\otimes u)\varphi_{r}\|_{L_{\lambda'}(\R^3_+)}^{1-\beta}\|\nabla ((u\otimes u)\varphi_{r})\|_{L_{\lambda'}(\R^3_+)}^\beta\big\|_{L_{\lambda'}(-r^2,0)}.
\end{align*}
Notice also that since $u$ has zero trace on the boundary we get
$$ \|u\otimes u\|_{L_{\lambda'}(B^{+}(r))}\leq c_{univ} r\|u\cdot\nabla u\|_{L_{\lambda'}(B^{+}(r))}.$$
Using the above two facts, one can show
\begin{equation}\label{localizedpressureest1}
D_{\lambda'}(p^{1},r)\leq c_{univ} \big(H_{\lambda'}(u,r)+ F_{\lambda', \beta}(u,r)\big).
\end{equation}
Applying Lemma \ref{fractionalsobolevwithTypeI}, we can estimate both $H_{\lambda'}(u,r)$ and $F_{\lambda', \beta}(u,r) $ to obtain
\begin{equation}\label{localizedpressureest2}
D_{\lambda'}(p^{1},r)\leq C{(A(u,r))}^{\theta}+C\big(E(u,r)+A(u,r)\big)^{\mu}.
\end{equation}
Here (and throughout this proof) $C(M,\lambda',s,\beta,\delta)>0$ is a constant that may change from line to line. 
From Theorem 1.5 of \cite{KS02} and by the Poincar\'{e} inequality, we have
$$\|\nabla u^{1}\|_{L_{\lambda'}(Q^{+}(r))}\leq C_{univ}\|u\otimes u\|_{L_{\lambda'}(Q^{+}(r))} $$
and
$$ \| u^{1}\|_{L_{\lambda'}(Q^{+}(r))}\leq C_{univ}r\|u\otimes u\|_{L_{\lambda'}(Q^{+}(r))}.$$
These estimates, together with an application of Lemma \ref{fractionalsobolevwithTypeI} to estimate $H_{\lambda'}(u,r)$, gives
\begin{equation}\label{v1lowerorderterms}
\frac{1}{r^{\frac{5-2\lambda'}{\lambda'}}}\Big(\|\nabla u^{1}\|_{L_{\lambda'}(Q^{+}(r))}+\frac{\|u^{1}\|_{L_{\lambda'}(Q^{+}(r))}}{r}\Big)\leq C(A(u,r))^{\frac{\theta}{\lambda'}}.
\end{equation}

\noindent\textbf{Step 2: the harmonic part of the pressure}\\
The strategy for estimating the harmonic part of the pressure is in the same spirit to that in \cite{Ser02}, though our context is different.
Define $u^{2}:=u-u^{1}$ and $p^{2}=p-p^{1}$. Then we have
\begin{align*}
&\partial_{t}u^{2}-\Delta u^2+\nabla p^{2}=0,\quad \nabla\cdot v^{2}=0\,\,\,\,\,\mbox{in}\,\,\,\,\,\, Q^{+}(\tfrac{r}{2}),\\
&v^{2}(x',0,t)=0.
\end{align*}
Using Remark \ref{constantsscaling} and $u^{2}= u-u^{1}$, we have
\begin{equation}\label{harmpreslocalreg}
\|\nabla p^{2}\|_{L^{t}_{\lambda'}L^{x}_{s}(Q^{+}(\frac{r}{4}))}\leq J_1+J_2.
\end{equation}
Here,
\begin{equation}\label{Jdef}
J_1:=\frac{r^{\frac{3}{s}+\frac{2}{\lambda'}}}{r^{\frac{5}{\lambda'}+1}}\Big(\frac{1}{r}\|u\|_{L_{\lambda'}(Q^{+}(\frac{r}{2}))}+\|\nabla u\|_{L_{\lambda'}(Q^{+}(\frac{r}{2}))}+\|p-(p)_{B^{+}(\frac{r}{2})}\|_{L_{\lambda'}(Q^{+}(\frac{r}{2}))}\Big)
\end{equation}
and
\begin{equation}\label{JJdef}
J_2:=\frac{r^{\frac{3}{s}+\frac{2}{\lambda'}}}{r^{\frac{5}{\lambda'}+1}}\Big(\frac{1}{r}\|u^1\|_{L_{\lambda'}(Q^{+}(\frac{r}{2}))}+\|\nabla u^1\|_{L_{\lambda'}(Q^{+}(\frac{r}{2}))}+\|p^1-(p^1)_{B^{+}(\frac{r}{2})}\|_{L_{\lambda'}(Q^{+}(\frac{r}{2}))}\Big).
\end{equation}
By applying H\"{o}lder's inequality one sees that
\begin{equation}\label{Jest}
J_1\leq r^{\frac{3}{s}+\frac{2}{\lambda'}-3}(E(u,r)^{\frac12}+A(u,{r})^{\frac12}+ (D_{\lambda'}(p,{r}))^{\frac{1}{\lambda'}}).
\end{equation}
Next, using \eqref{localizedpressureest2} and \eqref{v1lowerorderterms} we see that
\begin{equation}\label{JJest}
J_2\leq Cr^{\frac{3}{s}+\frac{2}{\lambda'}-3}\Big((A(u,{r}))^{\theta}+(E(u,r)+A(u,{r}))^{\mu}\Big)^{\frac{1}{\lambda'}}.
\end{equation}
Thus, 
\begin{equation}\label{nablap2est}
\|\nabla p^{2}\|_{L^{t}_{\lambda'}L^{x}_{s}(Q^{+}(\frac{r}{4}))}\leq r^{\frac{3}{s}+\frac{2}{\lambda'}-3}(E(u,r)^{\frac12}+A(u,{r})^{\frac12}+ (D_{\lambda'}(p,{r}))^{\frac{1}{\lambda'}})$$$$+Cr^{\frac{3}{s}+\frac{2}{\lambda'}-3}\Big((A(u,{r}))^{\theta}+(E(u,r)+A(u,{r}))^{\mu}\Big)^{\frac{1}{\lambda'}}.
\end{equation}
Notice that for $\tau\in (0,\frac{1}{4})$ we have
$$ (D_{\lambda'}(p^2, \tau r))^{\frac{1}{\lambda'}}\leq  \tau^{3-\frac{2}{\lambda'}-\frac{3}{s}}r^{3-\frac{2}{\lambda'}-\frac{3}{s}}\|\nabla p^{2}\|_{L^{x}_{s}L^{t}_{\lambda'}(Q^{+}(\frac{r}{4}))}. $$
Using this and \eqref{nablap2est}, we deduce that
\begin{equation}\label{p2estsmallercylinder}
D_{\lambda'}(p^2, \tau r)\leq {\tau}^{(3-\frac{3}{s}-\frac{2}{\lambda'})\lambda'}\Big((E(u,r)+A(u,{r}))^{\frac{\lambda'}2}+ D_{\lambda'}(p,{r})\Big)$$$$+C\tau^{(3-\frac{3}{s}-\frac{2}{\lambda'})\lambda'}\Big((A(u,{r}))^{\theta}+(E(u,r)+A(u,{r}))^{\mu}\Big).
\end{equation}
Finally, we note that $D_{\lambda'}(p, \tau r)\leq  \frac{C}{\tau^{5-2\lambda'}} D_{\lambda'}(p^{1}, r)+ CD_{\lambda'}(p^{2}, \tau r)$. We then use \eqref{localizedpressureest2} and the above estimate to get the desired conclusion.
\end{proof}
\subsection{Boundedness of scaled energy and pressure under Type I assumption}
We let $\lambda'$ be as in Lemma \ref{fractionalsobolevwithTypeI} and let $\frac{1}{\lambda}:=1-\frac{1}{{\lambda}^{'}}$.
In this subsection, we introduce the scaled quantities
\begin{equation}\label{scaledvelocityL3}
C(u,r):= \frac{1}{r^2}\int\limits_{Q^{+}(r)} |u|^{3} dxdt,
\end{equation}
\begin{equation}\label{scaledvelocityLlambda}
G_{\lambda}(u,r):= \frac{1}{r^{5-\lambda}} \int\limits_{Q^{+}(r)} |u|^{\lambda} dxdt.
\end{equation}
Now it suffices to gather the previous estimates in order to prove Theorem \ref{TypeIimpliesscaled}.

\begin{proof}[Proof of Theorem \ref{TypeIimpliesscaled}] 
Fix $\tau\in (0,\frac{1}{4})$ and let $r\in (0,1)$. By choosing an appropriate test function in the local energy inequality \eqref{localenergyequalityboundary} we obtain
\begin{equation}\label{localenergyest}
A(u, \tau r)+E(u, \tau r) \leq c(\tau)\big( C(u,r)^{\frac{2}{3}}+C(u,r)+(G_{\lambda}(u,r))^{\frac{1}{\lambda}}(D_{\lambda'}(p,r))^{\frac{1}{\lambda'}}\big)
\end{equation}
From Lemma \ref{interpolationwithTypeI} along with the fact that $\lambda<4$, we see that
\begin{equation}\label{L3typeI}
C(u,r)\leq C(M) A(u,r)^{(\frac12)_{+}}
\end{equation}
and
\begin{equation}\label{LlambdatypeI}
G_{\lambda}(u,r)\leq C(M,\lambda) A(u,r)^{(\frac\lambda2-1)_{+}}.
\end{equation}
Using \eqref{localenergyest} and the above two estimates, together with Young's inequality, we arrive at the following estimate
\begin{equation}\label{localenergyestyoungs}
A(u, \tau r)+E(u, \tau r)\leq \frac{1}{4}\big(A(u,  r)+E(u,  r)+D_{\lambda'}(p,r)\big)+C(\lambda,\tau,M).
\end{equation}
Let $\mathcal{E}(u,p,r):=A(u, r)+E(u,  r)+D_{\lambda'}(p,r).$
From \eqref{localenergyestyoungs} we have
\begin{equation}\label{Eest1}
\mathcal{E}(u,p,\tau r)\leq \frac{1}{4}\mathcal{E}(u,p,r)+C(\tau,\lambda,M)+D_{\lambda'}(p,\tau r).
\end{equation}
Applying Lemma \ref{pressuredecay} (with an appropriate choice of $s$ and $\tau(s,M,\lambda, \beta)$) and using Young's inequality, we obtain the bound
\begin{equation}\label{pressureiterative}
D_{\lambda'}(p,\tau r)\leq \frac{1}{4} \mathcal{E}(p,r)+ C(M,\beta,s,\lambda',\tau).
\end{equation}
Combining with \eqref{Eest1} then gives 
\begin{equation}\label{Eestiterative}
\mathcal{E}(u,p, \tau r)\leq \frac{1}{2} \mathcal{E}(u,p,r)+C(M,\beta,s,\lambda',\tau).
\end{equation}
An iterative argument then shows
\begin{equation}\label{Euniformbound}
\sup_{0<r\leq 1} \mathcal{E}(u,p,r) \leq C'\big(M,\beta,s,\lambda,\tau, \mathcal{E}(u,p,1)\big).
\end{equation}
Using Lemma 3.2 in \cite{Mik09} this also implies that
$$\sup_{0<r\leq 1} D_{\frac{3}{2}}(p,r)\leq C''\big(M,\beta,s,\lambda,\tau, A(u,1), E(u,1), D_{\frac{3}{2}}(p,1)\big).$$
Here, we have used that $\lambda'<\frac{3}{2}$ implies that $D_{\lambda'}(p,1)\leq D_{\frac{3}{2}}(p,1)$. 
This concludes the proof of Theorem \ref{TypeIimpliesscaled}.
\end{proof}

\section{Geometric regularity criteria near a flat boundary with vorticity alignment imposed on concentrating sets}
\label{sec.four}

The proof of Theorem \ref{CAnconcentratingglobal} reduces to showing that Proposition \ref{2Dreductionbycompactness} stated in the Introduction holds. Notice that \eqref{Morreybound} reads 
\begin{equation}\label{e.morreybound}
\sup_{r>0} \big\{A(\bar{u}, r)+E(\bar{u}, r)+ D_{\frac{3}{2}}(\bar{p}, r)\big\}\leq M'.
\end{equation}
After proving this proposition, we will then explain how it can be used to prove Theorem \ref{CAnconcentratingglobal}.

\begin{proof}[Proof of Proposition \ref{2Dreductionbycompactness}]
First note that by rescaling $$(\bar{\bar{u}}(x,t), \bar{\bar{p}}(x,t)):= (\sqrt{t_{0}} \bar{u}(\sqrt{t_{0}}x, t_{0}t), t_{0}\bar{p}(\sqrt{t_{0}}x, t_{0}t)$$
it suffices to consider $t_{0}\equiv 1$.\\
\noindent\textbf{Step 1: setting up the contradiction argument}\\
The proof is by contradiction. Suppose the statement is false, then there exists a sequence of mild solutions $(u^{(k)}, p^{(k)})$ to the Navier-Stokes equations on $\mathbb{R}^3_{+}\times (-\infty,0)$, which are also suitable weak solutions on $\mathbb{R}^3_{+}\times (-\infty,0)$, and $\gamma^{(k)}\uparrow \infty$ such that the following holds true:
\begin{equation}\label{TypeIboundsequences}
|u^{(k)}(x,t)|\leq \frac{M}{\sqrt{-t}}
\end{equation}
and
\begin{equation}\label{energyuniformbdd}
\sup_{R>0}\{A(u^{(k)},r)+E(u^{(k)},r)+D_{\frac{3}{2}}(p^{(k)},r)\}\leq M'.
\end{equation}
For $i=2,3$ and $\omega^{(k)}:= \nabla \times u^{(k)}$ we have \begin{equation}\label{sequencealmost2D}
\omega^{(k)}(x,-1)\cdot\vec{e}_{i}=0\,\,\,\,\textrm{for}\,\,\,\,x\in B^{+}(\gamma^{(k)}),\footnote{Recall that $\gamma^{(k)} \uparrow\infty$.}
\end{equation}
and  for all $k\in\mathbb{N}$, $u^{(k)}$ has a singular point at $(x,t)=(0,0)$.

\noindent\textbf{Step 2: passage to the limit}\\
Using \eqref{energyuniformbdd} and local  boundary regularity theory for the linear Stokes system (see \cite{Ser09}) we get that for all $r>0$ 
\begin{equation}\label{higherderivativebounds}
\sup_{k} \Big(\|u^{(k)}\|_{W^{2,1}_{\frac{9}{8}, \frac{3}{2}}(Q^{+}(r))}+ \|\nabla p^{(k)}\|_{L_{\frac{9}{8},\frac{3}{2}}(Q^{+}(r))}\Big)<\infty.
\end{equation}
From this, \eqref{energyuniformbdd} and the compact embedding 
$$W^{2,1}_{\frac98,\frac32}(B^+(K)\times(-1,0))\Subset C^0([-1,0];L^\frac98(B^+(K)))$$ 
(see \cite[Lemma 1, p.71]{simon1986compact}), we see that $u^{(k)}$ has a subsequence that converges to a limiting suitable weak solution $u^{(\infty)}$
\begin{equation}\label{strongconvergence}
\lim_{k\rightarrow\infty} \|u^{(k)}-u^{(\infty)}\|_{L_{3}(Q^{+}(r))}=0\,\,\,\textrm{for}\,\,\,\textrm{all}\,\,\,r>0.
\end{equation}
From \eqref{TypeIboundsequences}-\eqref{energyuniformbdd}, we see that
\begin{equation}\label{limitTypeI}
|u^{(\infty)}(x,t)|\leq \frac{M}{\sqrt{-t}}
\end{equation}
and 
\begin{equation}\label{limitscaledenergy}
\sup_{r>0}\{A(u^{(\infty)},r)+E(u^{(\infty)},r)+D_{\frac{3}{2}}(p^{(\infty)},r)\}\leq M'.
\end{equation}
Next note that the Duhamel formulation of the problem
\begin{align*}
\partial_{t}W-\Delta W+\nabla P= \nabla\cdot F\,\,\,\,\,\,\,(x,t)\in \mathbb{R}^3_{+}\times (0,\infty)\\
\nabla\cdot W=0,\,\,\,\,W|_{\partial \mathbb{R}^3_{+}}=0,\,\,\,\,\,\, W(x,0)= W_{0}(x)
\end{align*}
is continuous with respect to $L^{\infty}$ weak$^\star$ convergence of $W_{0}$ and $F$.\footnote{This can be inferred from arguments from p.555-556 of \cite{barker2015ancient}.} Consequently, from \eqref{TypeIboundsequences} the fact that $u^{(k)}$ are assumed to be mild solutions on $\mathbb{R}^3_{+}\times (0,\infty)$, we see that $u^{(\infty)}$ is a mild solution on $\mathbb{R}^3_{+}\times (-\infty,0)$. 
From \eqref{almost2D}, \eqref{higherderivativebounds} and the fact that $L_{\infty}$ mild solutions are smooth, we infer that
\begin{equation}\label{2Dinitialdata}
\omega^{(\infty)}(x,-1):= \nabla\times u^{(\infty)}(x,-1)= (\omega^{(\infty)}_{1},0,0)\,\,\,\,\textrm{for all}\,\,\,\,\,x\in\mathbb{R}^3_{+}.
\end{equation}
It is important to notice that \eqref{2Dinitialdata} holds pointwise, for all $x\in\R^3_+$, and not only almost everywhere. This is needed to get that $u^{(\infty)}$ is two-dimensional, see Step 3 below. 
Using the assumptions that $u^{(k)}$ has a singular point at the space-time origin for all $k$, \eqref{strongconvergence} and Lemma \ref{stabilitysingularpointshalfspace}, we see that $u^{(\infty)}$ has a singular point at $(x,t)=(0,0)$.

\noindent\textbf{Step 3: obtaining the contradiction}
\\
From \eqref{2Dinitialdata} and \eqref{limitTypeI}, we see that
\begin{equation}\label{u1laplace}
\Delta u^{(\infty)}_{1}=0,\,\,\,\ u^{(\infty)}_{1}|_{\partial \mathbb{R}^3_{+}=0}\,\,\,\,\, \mbox{and}\quad u^{(\infty)}_{1}\in L_{\infty}(\mathbb{R}^3_{+}).
\end{equation}
By the classical Liouville theorem for the Laplace equation, we infer $u^{(\infty)}_{1}\equiv 0$. Using this and \eqref{2Dinitialdata}, one can conclude that $u^{(\infty)}(x,-1)$ is two-dimensional, i.e.\begin{equation}\label{baru2Dinitialdata}
u^{(\infty)}(x,-1)=(0, u^{(\infty)}_{2}(x_{2}, x_{3}, -1), u^{(\infty)}_{3}(x_{2}, x_{3}, -1)).
\end{equation}
This, along with the scaled energy bound, implies that
$$ \sup_{r>0} \int\limits_{B^{+}_{\mathbb{R}^2}(r)} |u^{(\infty)}(y,-1)|^2 dy<\infty.$$
Hence, 
\begin{equation}\label{limitminus1}
u^{(\infty)}(x,-1) \in L_{\infty}(\mathbb{R}^3_+) \cap L_{2}(\mathbb{R}^2_+).
\end{equation}
From such initial data, we may construct local in time two-dimensional mild solutions \cite{baejin}. Similar arguments to \cite{barker2015ancient} (specifically p.566-568) show that this solution also satisfies the three-dimensional Duhamel formula.
 Using that $u^{(\infty)}$ is a mild solution on $\mathbb{R}^3_{+}\times (-\infty,0)$ and that $u^{(\infty)}$ is bounded locally in space-time, we may use the uniqueness of $L_{\infty}$ mild solutions to the Navier-Stokes equations in the case where the domain is a half-space (Theorem 1.3 in \cite{baejin}). One then concludes that when $t\in (-1,0)$, $u^{(\infty)}(x,t)$  spatially depends only on $(x_{2}, x_{3}).$ 
 This, in conjunction with \eqref{limitscaledenergy}, implies that for $t\in [-1,0)$ we have
 $$ \sup_{r>0}\int\limits_{B_{\mathbb{R}^2}^{+}(r)} |u^{(\infty)}(y,t)|^2 dy<\infty.$$
 Hence,
 \begin{equation}\label{limitfiniteKE}
u^{(\infty)}\in L^{\infty}_{t}(-1,0; L_{2}(\mathbb{R}^2_{+})).
 \end{equation}
 Using \eqref{limitTypeI}, we see that $$u^{(\infty)}\otimes u^{(\infty)} \in L_{2,loc}^{t}([-1, 0); L_{2}^{x}(\mathbb{R}^2_+)).$$
 This implies that $u^{(\infty)}$ satisfies the energy inequality. Thus $u^{(\infty)}$ is a two-dimensional weak Leray-Hopf solution for $t\in [-1,0)$, with initial data $u^{(\infty)}(\cdot,-1)$. It is known that two-dimensional weak Leray-Hopf solutions are infinitely differentiable in space and time for all positive times \cite{CFbook}, 
 hence $u^{(\infty)}$ cannot have a singular point at $(x,t)\in (0,0)$. Thus, we have obtained the desired contradiction.
\end{proof}
\subsection*{Proof of Theorem \ref{CAnconcentratingglobal}}
\begin{proof}
We will present the proof of \eqref{CAconditionconcentrate} here, whilst commenting on the proof of \eqref{e.vortbddreg} in Section \ref{improveremark}. 
We will take $x_{0}=0$ . The case where $x_{0}$ is in the interior can be proven similarly. Without loss of generality take $T>1$. For the proof of \eqref{CAconditionconcentrate} we will assume for contradiction that $(x,t)=(0,T)$ is a singular point of $u$.

\noindent\textbf{Step 1: Local Morrey bounds}\\
Clearly for $t\in(0,T)$, $u$ satisfies the energy inequality
\begin{equation}\label{uenergyinequality}
\|u(\cdot,t)\|_{L_{2}(\mathbb{R}^3_{+})}^2+\int\limits_{0}^{t}\int\limits_{\mathbb{R}^3_{+}} |\nabla u|^2 dxds\leq \|u_{0}\|^2_{L_{2}(\mathbb{R}^3_{+})}.
\end{equation}
Let $p$ be the pressure associated to $u$.
Using \eqref{uenergyinequality}, maximal regularity for the linear Stokes equations \cite{GS91} and  estimates for the Stokes semigroup in the half-space, 
one can infer that
\begin{equation}\label{morreyscale1}
\int\limits_{T-1}^{T}\int\limits_{B^{+}(1)} |p-(p)_{B^{+}(1)}|^{\frac{3}{2}}+|\nabla u|^2 dxds+ \sup_{T-1<s<T} \|u(\cdot,s)\|_{L_{2}(B^{+}(1))}^2 \leq C(u_0).
\end{equation}
Notice that the left hand side is bounded in terms of the initial data $u_0$ only. 
One can readily show that $u$ is a suitable weak solution to the Navier-Stokes equations on $\mathbb{R}^3_{+}\times (0,T)$. Using this, together with \eqref{morreyscale1} and the ODE blow-up rate \eqref{TypeIhalf}, we may apply Theorem \ref{TypeIimpliesscaled} 
to deduce that
\begin{multline}\label{Morreysmallscaleu}
\sup_{0<r<1}\Bigg(\frac{1}{r^2}\int\limits_{T-r^2}^{T}\int\limits_{B^{+}(r)} |p-(p)_{B^{+}(r)}|^{\frac{3}{2}} dxds+ \frac{1}{r}\int\limits_{T-r^2}^{T}\int\limits_{B^{+}(r)} |\nabla u|^2 dxds\\
+\frac{1}{r}\sup_{T-r^2<s<T}\|u(\cdot,s)\|_{L_{2}(B^+(r))}^2\Bigg)\leq M'(u_{0}, M).
\end{multline}
With applying the previous Proposition in mind, we now select $\delta= \gamma(M,M')$. Here $\delta$ is from the definition of $\mathcal{C}_{\delta, x_{0}}$ in Theorem \ref{CAnconcentratingglobal} and $\gamma$ is from Proposition \ref{2Dreductionbycompactness}.

\noindent\textbf{Step 2: zoom in on the singularity and a priori bounds}\\
Let $R^{(n)}\downarrow 0$ 
and rescale
$$u^{(n)}(y,s):= R^{(n)} u(R^{(n)}y, T+ (R^{(n)})^2 s)\,\,\,\,\textrm{and}\,\, p^{(n)}(y,s):=(R^{(n)})^{2} p(R^{(n)}y, T+  (R^{(n)})^2s), $$
for all $y\in\R^3_+$ and $s\in (-T/(R^{(n)})^{2},0)$. 
From \eqref{Morreysmallscaleu}, we have
\begin{equation}\label{apriori1main}
\sup_{n}\sup_{0<r<1/R^{(n)}}\{A(u^{(n)},r)+E(u^{(n)},r)+ D_{\frac{3}{2}}(p^{(n)}, r)\}\leq M'.
\end{equation}
From \eqref{TypeIhalf} we have
\begin{equation}\label{apriori2main}
\sup_{n}\sup_{-T/(R^{(n)})^2<s<0} (-s)^{\frac{1}{2}}\|u^{(n)}(\cdot,s)\|_{L_{\infty}}\leq M
\end{equation}
Since $u^{(n)}$ is a strong solution on $\mathbb{R}^3_{+} \times (-T/(R^{(n)})^2),0)$ we can use \eqref{apriori2main} and the same arguments in \cite{barker2015ancient}\footnote{See the proof of Theorem 1.3 in \cite{barker2015ancient}.} to infer that for any $\ep>0$ and $l,m\in \mathbb{N}$ we have
\begin{equation}\label{apriori3mainbis}
\sup_{n} \sup_{-T/(R^{(n)})^2<s<\ep} \|\partial_{t}^{l}\nabla^{m} u^{(n)}(\cdot,s)\|_{L_{\infty}}<\infty.
\end{equation} 

\noindent\textbf{Step 3: passage to the limit}\\
Using the above a priori estimates and Lemma \ref{stabilitysingularpointshalfspace} we obtain a limiting  suitable weak solution $(u^{(\infty)}, p^{(\infty)})$ to the Navier-Stokes equations on $\mathbb{R}^3_{+}\times (-\infty,0)$ with no-slip boundary condition and such that
\begin{align}
u^{(\infty)}\quad\mbox{has a singular point at}\quad (x,t)=(0,0)\notag\\
\sup_{0<r<\infty}\{ E(u^{(\infty)},r)+A(u^{(\infty)},r)+D_{\frac{3}{2}}(p^{(\infty)},r)\}\leq M'\notag\\
\sup_{-\infty<t<0} \sqrt{-t}\|u^{(\infty)}(\cdot,t)\|_{L_{\infty}(\mathbb{R}^3_{+})}\leq M
\end{align}
Moreover, arguing as in Proposition \ref{2Dreductionbycompactness}, we infer that $(u^{(\infty)}, p^{(\infty)})$ is a mild solution on $\mathbb{R}^3_{+}\times (-\infty,0)$.  
As was previously mentioned, the results from \cite{barker2015ancient} imply that 
\begin{equation}\label{mbassmooth}
\sup_{n} \sup_{-\infty<s<\ep} \|\partial_{t}^{l}\nabla^{m} u^{(\infty)}(\cdot,s)\|_{L_{\infty}}<C(\epsilon,M).
\end{equation}\\
\noindent\textbf{Step 4: reduction to Proposition \ref{2Dreductionbycompactness} when $\omega^{(\infty)}|_{\partial\mathbb{R}^3}$ is non-zero}\\
We now encounter two cases, namely
\begin{enumerate}
\item $\omega^{(\infty)}|_{\partial\mathbb{R}^3_{+}}(\cdot,-t_{0})\neq 0$ for some $t_{0}\in(0,\infty)$.
\item $\omega^{(\infty)}|_{\partial\mathbb{R}^3_{+}}(\cdot,t)=0$ for all $t\in(-\infty,0)$.
\end{enumerate}
We first treat case 1) by using \eqref{CAconditionconcentrate} to show \eqref{almost2D} holds true. This essentially follows from the same reasoning as in \cite{GM11}, but we provide full details for completeness. 
Using \eqref{apriori3mainbis} we obtain that 
\begin{equation}\label{uniformconvergence}
\omega^{(n)}(x, -t_{0})\rightarrow \omega^{(\infty)}(x,-t_{0})\,\,\,\,\,\,in\,\,C\big(\overline{B^{+}}(\gamma(M,M')\sqrt{t_{0}})\big). 
\end{equation}
Define $$S^{+}_0:=\{x\in B^{+}(\gamma(M,M')\sqrt{t_{0}}): |\omega^{(\infty)}(x,-t_{0})|>0\}.$$ 
Now, fix any $x\in S^{+}_0$. Then there exists $K\in\N\setminus\{0\}$ such that $$|\omega^{(\infty)}(x,-t_{0})|>K^{-1}.$$ 
Using \eqref{uniformconvergence}, this implies that for all $n$ sufficiently large we have
$$|\omega^{(n)}(x,-t_{0})|\geq \frac{1}{2K}\Rightarrow |\omega(R^{(n)}x,T-(R^{(n)})^2t_{0})|\geq \frac{1}{2K(R^{(n)})^2}\rightarrow\infty.$$ 
Hence, for all $x\in S^{+}_0$ we see that for $n$ sufficiently large we have
$$(R^{(n)}x, T-(R^{(n)})^2t_{0}) \in \Omega_{d}\cap C_{ \delta, x_{0}}.$$ 
Writing $\xi^{(n)}:= \frac{\omega^{(n)}}{|\omega^{(n)}|}$ and using the continuous alignment condition \eqref{CAconditionconcentrate}, we see that for any $x$ and $y$ in $S^{+}$ there exists $n$ large enough such that
$$|\xi^{(n)}(x,-t_{0})-\xi^{(n)}(y,-t_{0})|\leq C\eta(R^{(n)}|x-y|)\rightarrow 0.$$
Thus, the vorticity direction for $\omega^{(\infty)}(\cdot,-t_{0})$ points in one direction in $B^{+}(\gamma(M,M')\sqrt{t_{0}})$. By the no-slip boundary condition on $\partial\R^3_+$, we obtain that $\omega^{(\infty)}(x,-t_{0})\cdot\vec{e}_3=0$. Here we have crucially used that in case 1) $\omega^{(\infty)}(\cdot,-t_{0})$ does not vanish on the boundary, implying the vorticity direction is well defined for some spatial point on the boundary. By a suitable rotation we can then assume without loss of generality that in $B^{+}(\gamma(M,M')\sqrt{t_{0}})$, $$\omega^{(\infty)}(x,-t_{0})\cdot \vec{e}_{i}=0$$ for $i=2,3$. We conclude case 1) by relying on Proposition \ref{2Dreductionbycompactness} to show that $(0,0)$ is a regular point of $u^{(\infty)}$. This is a contradiction with the fact that $u^{(\infty)}$ has a singular point at $(x,t)=(0,0)$.
 
\noindent\textbf{Step 5: Treating the case $\omega^{(\infty)}|_{\partial\mathbb{R}^3_{+}}\equiv 0$}\\
Notice that for the case $\omega^{(\infty)}|_{\partial\mathbb{R}^3_{+}}(\cdot,t)\equiv 0$ for all $t\in(-\infty,0)$, the vorticity direction is not defined on the boundary and hence the final implication in the previous step (inferring $\omega^{(\infty)}\cdot \vec{e}_{3}$ is zero in some ball using the no-slip boundary condition) does not apply here. This possibility seems to have  not been considered in \cite{GM11} and \cite{giga2014liouville}. To overcome this difficulty we will extend the solution $u^{(\infty)}$ to the whole space and then apply a theorem regarding unique continuation for parabolic operators from \cite{ESS2003} to the vorticity equation.  This will show that $u^{(\infty)}\equiv 0$ in $\mathbb{R}^3_{+}\times (-\infty,0)$, which contradicts the fact that $u^{(\infty)}$ has a singular point at $(x,t)=(0,0)$.

First, we recall from \eqref{mbassmooth} that $u^{(\infty)}$ is infinitely differentiable in space-time with bounded derivatives on $\mathbb{R}^3_{+}\times (-\infty, \epsilon)$. Notice that $\nabla\cdot u^{(\infty)}=0$, $u^{(\infty)}|_{\partial\mathbb{R}^3_{+}}=0$ and the assumption that $\omega^{(\infty)}|_{\partial\mathbb{R}^3_{+}}(\cdot,t)\equiv 0$ for all $t\in (-\infty,0)$ implies the following. Namely, for every multiindex $\alpha$ we have 
\begin{equation}\label{vanishingallorders}
D^{\alpha} u^{(\infty)}(x_{1}, x_{2}, 0, t)=0\,\,\,\,\,\textrm{for}\,\,\textrm{all}\,\,\,(x_{1}, x_{2})\in\mathbb{R}^2\,\,\,\,\textrm{and}\,\,\,\,t\in (-\infty,0).
\end{equation}
Proceeding as in \cite{bae2008regularity}, we use \eqref{vanishingallorders} to show that there is a well defined extension of $(u^{(\infty)}, p^{(\infty)})$ to $\mathbb{R}^3\times (-\infty,0)$ denoted by $\big((u^{(\infty)})^{*}, (p^{(\infty)})^{*}\big)$ and given by
\begin{equation}\label{extend}
\big((u^{(\infty)})^{*}, (p^{(\infty)})^{*}\big)(x,t):=\big(u^{(\infty)}_{1},u^{(\infty)}_{2},-u^{(\infty)}_{3}, p^{(\infty)}\big)(x_{1},x_{2}, -x_{3},t)\,\,\,\,\,\textrm{for}\,\,\,\,x_{3}<0.
\end{equation}
Moreover, \eqref{vanishingallorders} and the arguments in \cite{bae2008regularity} give that $\big((u^{(\infty)})^{*}, (p^{(\infty)})^{*}\big)$ is a solution to the Navier-Stokes equations in $\mathbb{R}^3\times (-\infty,0)$. Now, we may use the smoothness given by \eqref{mbassmooth} to deduce for the vorticity $(\omega^{(\infty)})^{*}:= \nabla\times (u^{(\infty)})^{*}$ that
\begin{equation}\label{parabolicinequality}
|\partial (\omega^{(\infty)})^{*}-\Delta (\omega^{(\infty)})^{*}|\leq C(M,\epsilon)(|(\omega^{(\infty)})^{*}|+|\nabla(\omega^{(\infty)})^{*}|)\,\,\,\,\,\textrm{in}\,\,\,\,\,\mathbb{R}^3\times (-\infty,\epsilon).
\end{equation}
This, as well as \eqref{mbassmooth} and \eqref{vanishingallorders}, allows the applicability of a theorem regarding unique continuation for parabolic operators (Theorem 4.1 in \cite{ESS2003}). This then gives that $(\omega^{(\infty)})^{*}\equiv 0$ in $\mathbb{R}^3\times (-\infty,0)$ (since $\epsilon$ was chosen arbitrarily). One then concludes that $u^{(\infty)}$ is zero by using the classical Liouville theorem for harmonic functions, along with the no-slip boundary condition.
\end{proof}

\section{Further improvements and remarks}\label{improveremark}

\subsection{Remark regarding proof of \eqref{e.vortbddreg}}
Let us sketch how to prove \eqref{e.vortbddreg} in Theorem \ref{CAnconcentratingglobal}. We  assume for contradiction that $(x,t)=(0,T)$ is a singular point of $u$. 
In this case take  $\delta=\gamma(M,M')$, $R^{(n)}:= \sqrt{T- t^{(n)}}$
and rescale
$$U^{(n)}(y,s):= R^{(n)} u(R^{(n)}y, T+ (R^{(n)})^2 s)\,\,\,\,\textrm{and}\,\, P^{(n)}(y,s):=(R^{(n)})^{2} p(R^{(n)}y, T+  (R^{(n)})^2s), $$
for all $y\in\R^3_+$ and $s\in (-T/(R^{(n)})^{2},0)$. From the assumption in \eqref{e.vortbddreg}, we see that for $\Omega^{(n)}:=\nabla \times U^{(n)}$ we have
\begin{equation}\label{vorttozero}
\|\Omega^{(n)}(\cdot,-1)\|_{L_{\infty}(B^{+}(\gamma(M,M'))}\leq C R_{n}^2 \downarrow 0.\end{equation}
Repeating Steps 1-3 in the same way, we end up with a limiting solution $(U^{(\infty)}, P^{(\infty)})$ with the properties as described in Step 3 for $(u^{(\infty)}, p^{(\infty)})$. Additionally, from \eqref{vorttozero} we get that
$\Omega^{(\infty)}(\cdot,-1)=0$ in $B^{+}(\gamma(M,M'))$. This last fact allows us to immediately  apply Proposition \ref{2Dreductionbycompactness} to infer that $(0,0)$ is a regular point of $U^{(\infty)}$, which again gives a contradiction. 

\subsection{Improvements to Theorem \ref{CAnconcentratingglobal} for the whole-space case}

 The  above proof of Theorem \ref{CAnconcentratingglobal} demonstrates a difficulty specific to the Navier-Stokes equations in the half-space. Namely, if the vorticity direction is constant for the Navier-Stokes equations in the half-space, we  can only perform rotations that leave the half-space invariant. This creates difficulties when attempting to reduce to the situation described in Proposition \ref{2Dreductionbycompactness}. We overcame these difficulties in Steps 4 and 5.

 In comparison,  the above issue is not problematic in the whole-space setting. Consequently, for the whole-space setting we can prove a stronger result. In fact, this improves the result obtained in \cite{GM11}, as we discuss below. 
\begin{theorem}\label{CAnconcentratingglobalws}
Suppose $u$ is a mild solution to the Navier-Stokes equations in $\mathbb{R}^3\times (0,T)$  with divergence-free initial data $u_{0}\in C^{\infty}_{0}(\mathbb{R}^{3}).$ Furthermore, suppose that for $(x,t)\in \mathbb{R}^3\times (0,T)$:
\begin{equation}\label{TypeImaintheoremws}
|u(x,t)|\leq \frac{M}{\sqrt{T-t}}.
\end{equation}
Let $t^{(n)}\uparrow T$, $x_{0}\in\mathbb{R}^3$, $d>0$ and $\delta>0$. Define $\omega=\nabla\times u$, $\xi:= \frac{\omega}{|\omega|}$,  
$$\Omega_{d}:=\{(x,t)\in \mathbb{R}^3\times(0,T): |\omega(x,t)|>d\}$$ and

$$ \mathcal{C}_{t^{(n)}, \delta, x_{0}}:= \bigcup_{n}\{(x,t^{(n)}): |x-x_{0}|< \delta \sqrt{T-t^{(n)}}\}.$$
Let $\eta: \mathbb{R} \rightarrow \mathbb{R}$ be a continuous function with $\eta(0)=0$.\\
Under the above assumptions, there exists $\delta(M,u_{0})>0$ such that the following holds true:
\begin{align}
\sup_{(x,t)\in \mathcal{C}_{t^{(n)}, \delta, x_0}} |\omega(x,t)|<\infty\Rightarrow\,\,\,\,\,\,(x_{0}, T)&\quad\mbox{is a regular point of}\quad  u,
\label{e.vortbddregws}\\ 
\begin{split}\label{CAconditionconcentratews}
|\xi(x,t)-\xi(y,t)|\leq C\eta(|x&-y|)\,\,\,\,\,\,\,\,\,\,\,\,\,\,\,\textrm{in}\,\,\,\,\,\Omega_{d}\cap \mathcal{C}_{t^{(n)}, \delta, x_0}\\
&\Downarrow\\
(x_{0}, T)\quad \mbox{is a}\ & \mbox{regular point of}\quad u.
\end{split}
\end{align}
\end{theorem}

The proof is essentially the same to that of Theorem \ref{CAnconcentratingglobal}, hence we omit it. The difference is that one takes $R^{(n)}:= \sqrt{T- t^{(n)}}$
and then rescales in the following way
$$u^{(n)}(y,s):= R^{(n)} u(R^{(n)}y, T+ (R^{(n)})^2 s)\,\,\,\,\textrm{and}\,\, p^{(n)}(y,s):=(R^{(n)})^{2} p(R^{(n)}y, T+  (R^{(n)})^2s), $$
for all $y\in\R^3_+$ and $s\in (-T/(R^{(n)})^{2},0)$.
\begin{remark}\label{compareGigaMiura}
Consider $u$, $\xi$, $u_{0}$, $T$ and $t^{(n)}$ be as in the above theorem, but assume that $u$ first loses smoothness at time $T$.  Let $\eta$ be any fixed modulus of continuity with $\eta(0)=0$. The result in \cite{GM11} implies that $\xi$ breaks the modulus of continuity $\eta$ on the set 
\begin{equation}\label{gmmodulus}
\Omega_{d}\cap (\mathbb{R}^3\times (0,t^{(n)}))\end{equation}
for all $n$ sufficiently large. Since a finite blow-up time $T$ implies the existence of a singular point $x_{0}$, our  Theorem above implies a more precise description of the vorticity direction near the blow-up time.  Namely, the modulus of continuity $\eta$ is broken on the set 
\begin{equation}\label{refinemod}
\big\{(x,t^{(n)}): |x-x_{0}|< \delta \sqrt{T-t^{(n)}}\big\}\cap \Omega_{d}
\end{equation}
 for all sufficiently large $n$. Due to \eqref{e.vortbddregws} in the above theorem, we see that for all sufficiently large $n$, \eqref{refinemod} is non empty and is a strict subset of \eqref{gmmodulus}.   
\end{remark}

\section{Geometric regularity criteria near a flat boundary: local setting}
\label{sec.local}

In this section, we prove versions of the geometric regularity criteria under local assumptions on the solution.

The proof of the geometric criteria in \cite[Theorem 1.3]{giga2014liouville} or Theorem \ref{CAnconcentratingglobal} above 
cannot be easily localised in the half-space for the  following reason. In the local setting close to boundaries, we can have boundedness of $u$ but unboundedness of the vorticity, as can be seen from the construction of bounded shear flows with unbounded derivative in the half-space \cite{Kang05,SSv10}.  Hence, there is no a priori bound such as \eqref{apriori3mainbis} that ensures that the vorticity direction converges in the space of continuous functions. But this is a crucial ingredient to use \eqref{vorticitydirection} and to thus show that the limiting flow (resulting from the rescaling procedure) is two-dimensional. This is in strong contrast with the local setting \emph{away from boundaries}, where the proof above applies mutatis mutandis. 

We show here that in spite of the aforementioned difficulties, Theorem \ref{Gigaetal} can be localised. As far as we know, this is a new result for the half-space. For the local result in the interior case, we refer to \cite{GM11}.

\begin{theorem}[localised geometric regularity criteria]\label{theo.locgeomcrit}
Let $(u,p)$ be a suitable weak solution to the Navier-Stokes equations in $Q^+(1)$. Assume the Type I condition
\begin{equation}\label{ODEratelocal}
\sup_{t\in(-1,0)}(-t)^{\frac{1}{2}}\|u(\cdot,t)\|_{L_{\infty}(B^+(1))}<\infty.
\end{equation}
Let $d$ be a positive number and let $\eta$ be a continuous function on $[0,\infty)$ satisfying $\eta(0)=0$. Assume that $\eta$ is a modulus of continuity in the $x$ variables for the vorticity direction $\xi=\omega/|\omega|$, in the sense that
\begin{equation}\label{vorticitydirectionloc}
|\xi(t,x)-\xi(t,y)|\leq \eta(|x-y|)\,\,\,\,\,\,\textrm{for}\,\,\,\,\,(x,t),(y,t)\in \Omega_{d},
\end{equation}
where $\Omega_{d}=\{(x,t)\in \overline{B^+}(1)\times (-1,0): d<|\omega(x,t)|<\infty\}$.\footnote{Notice that our definition of $\Omega_d$ here is slightly different from the one in the global setting. Indeed, in order to deduce that the blow-up flow is two-dimensional we need to use the no-slip boundary condition, as in Step 4 in Section \ref{sec.four}. Hence, we require that the continuous alignment condition is true up to the boundary. In the global setting we do not need to assume this explicitly, since by \eqref{apriori3mainbis} the flow is smooth up to the boundary.}\\
Then 
$u\in L_\infty(Q^+(\frac12))$.
\end{theorem}

For the whole proof, we assume that $(u,p)$ is a suitable weak solution in $Q^+(1)$. Though our goal is to show that any point in $Q^+(\frac12)$ is regular it is enough to show that $(0,0)$ is a regular point. Indeed, regularity is clear for any time $t<0$ by the Type I condition \eqref{ODEratelocal} and one can reduce to proving that $(0,0)$ is regular by translation and scaling. 
Hence we assume that $(0,0)$ is a space-time blow-up point. Our aim is to reach a contradiction under the continuous alignment condition \eqref{vorticitydirectionloc}. 
In addition to using the stability of singularities as in the proof of Theorem \ref{CAnconcentratingglobal} above, we introduce two further techniques:
\begin{enumerate}
\item We rely on the localisation procedure from \cite{neustupaonecomponent}, which was subsequently also utilized in \cite{AlbrittonBarker2018local} to get half-space mild bounded ancient solutions from singular (local) suitable weak solutions. The role of this step is to obtain a limiting solution that is mild from the rescaled solution. We emphasize that the localisation serves this sole purpose, but it is crucial. Indeed  
showing that the flow is two-dimensional relies upon the uniqueness of mild solutions in $L_{\infty}(\mathbb{R}^3_{+}\times (0,T))$.\label{itemuniquenessmild}
\item We go around the use of strong convergence for the vorticity by using Egoroff's theorem.
\end{enumerate}
\noindent {\bf Step 1: localisation}\\
We rely on following lemma, which follows the presentation from \cite[Lemma 2.4]{AlbrittonBarker2018local}. To the best of the authors' knowledge such localisation procedures first appeared for the Navier-Stokes equations in \cite[Lemma 1]{neustupaonecomponent}.

\begin{lemma}[Regular annulus lemma]\label{lem.regan}
There exists $0<r_1<r_2\leq \frac14$ and $0<\delta_1<\delta_2< 1$ such that 
\begin{align*}
&u\in L_\infty\big(Q^+(r_2,\delta_2)\setminus Q^+(r_1,\delta_1)\big),\\
&(u,p)\in W^{2,1}_{p,\frac32}\times W^{1,0}_{p,\frac32}\big(Q^+(r_2,\delta_2)\setminus Q^+(r_1,\delta_1)\big),\qquad\forall p\geq 1,
\end{align*}
where the domain $Q^+(r,\delta)$ for $r\in(0,1]$ and $\delta\in[-1,0)$ is defined by
\begin{equation*}
Q^+(r,\delta)=B^+(r)\times (-\delta,0).
\end{equation*}
\end{lemma}

We fix two cut-off functions $\varphi_1\in C_{0}^{\infty}(\mathbb{R}^3)$ and $\varphi_2\in C^\infty_0(-\infty,\infty)$ such that
\begin{align*}
&0\leq\varphi_1,\, \varphi_2\leq 1,\,\,\,\,\, \varphi_1=1\,\,\,\,\textrm{on}\,\,\,\,B(r_1)\,\,\,\,\,\textrm{and}\,\,\,\,\,\varphi_1=0\,\,\,\,\,\textrm{off}\,\,\,\,\, B(r_2),\\
&\varphi_2=1\,\,\,\,\textrm{on}\,\,\,\,(-\delta_1,\delta_1)\,\,\,\,\,\textrm{and}\,\,\,\,\,\varphi_2=0\,\,\,\,\,\textrm{off}\,\,\,\,\, (-\delta_2,\delta_2).
\end{align*}
We then define $\varphi(x,t):=\varphi_1(x)\varphi_2(t)$ for all $(x,t)\in\R^3\times(-\infty,\infty)$. 
Let $B=B(u\cdot\nabla \varphi)$ be the Bogovskii operator defined on $B^+(r_2)\setminus B^+(r_1)$.\footnote{See \cite{Bogovski}, as well as Chapter III of \cite{Galdibook}.}  
In particular, for almost every $t\in (-1,0)$, $B(\cdot,t)$ solves 
\begin{equation}\label{e.Bogovsolve}
\nabla\cdot B(\cdot,t)= u(\cdot,t)\cdot\nabla\varphi(\cdot,t),\qquad\qquad B(\cdot,t)|_{\partial (B^+(r_2)\setminus B^+(r_1))}=0.
\end{equation}
Notice that Lemma \ref{lem.regan} implies $u\cdot\nabla\varphi\in W^{2,1}_{p,\frac32}(\R^3_+\times(-1,0))$. Moreover $u\cdot\nabla\varphi\in L_\frac32(-1,0;\mathring{W}^1_p(B^+(r_2)\setminus B^+(r_1)))$ and $\partial_t(u\cdot\nabla\varphi)\in L_{\frac{3}{2}}(-1,0; L_{p}(B^{+}(r_{2})\setminus B^{+}(r_{1}))$ 
for all finite $p\geq 1$. 
Since $u$ has zero trace on $\partial B^{+}(1)\cap\partial\mathbb{R}^3_{+}$, one sees that $u\cdot\nabla\varphi$ has zero average when integrated on $B^{+}(r_{2})\setminus B^{+}(r_{1})$. Hence, the operator $B$ is well defined. In view of \cite[Theorem III.3.3 and Exercise III.3.7]{Galdibook} we have
\begin{align}\label{e.bogoest}
\|\nabla B\|_{L_{\frac32}(-1,0;\mathring{W}^{1}_{p}(B^+(r_2)\setminus B^+(r_1)))}\leq\ & C(r_{1},r_{2},\delta_1,\delta_2,p)\|u\|_{W^{2,1}_{p,\frac32}(Q^+(r_2,\delta_2)\setminus Q^+(r_1,\delta_1))},\\
\label{e.bogest2}
\|\partial_{t} B\|_{L_{\frac32}(-1,0;L_{p}(B^+(r_2)\setminus B^+(r_1)))}\leq\ & C(r_{1},r_{2},\delta_1,\delta_2,p)\|u\|_{W^{2,1}_{p,\frac32}(Q^+(r_2,\delta_2)\setminus Q^+(r_1,\delta_1))}
\end{align}
for all $p\in(1,\infty)$. 
We can extend $B$ by zero on 
\begin{equation*}
\Big[\big((\R^3_+\setminus B^+(r_2))\cup B^+(r_1)\big)\times (-1,0)\Big]\cup\big[\R^3_+\times(-\infty,-1)\big]
\end{equation*}
and still denote this extension by $B$. It follows from \eqref{e.bogoest} and \eqref{e.bogest2} and the fact that $B$ is compactly supported that 
\begin{equation}\label{e.BdtBbis}
B\in W^{2,1}_{p,\frac32}(\mathbb{R}^3_{+}\times (-1,0)),
\end{equation}
for any $p\in(1,\infty)$.

On $\R^3_+ \times (-1,0)$ we define $U= \varphi u-B$ and $P=\varphi p$. 
As a result  of \eqref{ODEratelocal} and \eqref{e.BdtBbis}, the Sobolev embedding $W^{2,1}_{p,\frac32}(\mathbb{R}^3_{+}\times (-1,0))\hookrightarrow L_{\infty}(\mathbb{R}^3_{+}\times (-1,0))$ for $p>\frac92$ and the fact that $U$ is compactly supported, we have 
\begin{equation}\label{e.typeIU}
\sup_{-1<s<0} (-s)^{\frac{1}{2}}\|U(\cdot,s)\|_{L_{\infty}(\mathbb{R}^3_{+})}.
\end{equation}
Furthermore, we have 
$W^{2,1}_{2, \frac{3}{2}}(\mathbb{R}^3_{+}\times (-1,0))\hookrightarrow W^{1,0}_{2}(\mathbb{R}^3_{+}\times (-1,0))$. So using that $U$ has compact support we can infer that
\begin{equation}\label{Uenergyclass}
U\in W^{1,0}_{2}(\mathbb{R}^3_{+}\times (-1,0))\cap C_{w}([-1,0]; L_{p}(\mathbb{R}^3_{+})) 
\end{equation}
for all $1\leq p\leq 2$. 
Moreover, defining $r':= \min(\sqrt{\delta_1}, r_{1})$, it is clear from the properties of the cut-off function $\varphi$ that we have
\begin{equation}\label{locallyequal}
(U,P)=(u,p)\,\,\,\,\,\textrm{on}\,\,\,\,\,Q^{+}(r')
\end{equation}
From this,  the Type I bound \eqref{ODEratelocal} for $u$ and Theorem \ref{TypeIimpliesscaled} applied to $(u,p)$, we have
\begin{equation}\label{e.aprioriU}
\sup_{0<r<r'}\big\{A(U,r)+E(U,r)+D_{\frac32}(P,r)\big\}<\infty.
\end{equation}
We also have that $(U,P)$ solves the following Navier-Stokes equations 
\begin{align}\label{e.Ulocalisedequation}
&\partial_{t} U+U\cdot\nabla U-\Delta U+\nabla P=F,\,\,\,\,\,\, \nabla\cdot U=0\,\,\,\mbox{in}\,\,\, \R^3_+\times (-1,0),
\end{align}
with forcing term
\begin{align}\label{e.defforcing}
\begin{split}
F=\ &u(\partial_t-\Delta)\varphi-2\nabla\varphi\cdot\nabla u+(\varphi^2-\varphi)u\cdot\nabla u+\varphi u\cdot\nabla\varphi u+p\nabla\varphi\\
&-(\partial_t-\Delta)B-B\cdot\nabla (\varphi u)-\varphi u\cdot\nabla B+B\cdot\nabla B
\end{split}
\end{align}
no-slip boundary condition $U|_{\partial\mathbb{R}^3_{+}}=0$ and initial data $U(\cdot,-1)=0$. The key point is that $u$ is subcritical on the support of $\varphi$, so that the source term $F$ is subcritical. This explains our choice of cut-off function $\varphi$. Notice that in view of \eqref{e.BdtBbis}, Lemma \ref{lem.regan}, the fact that $B$ is complactly supported in $B^+(r_2)\setminus B^+(r_1)\times (-\delta_2,0)$ and known parabolic embeddings \cite{Ser06}, we infer that
\begin{equation}\label{e.intF}
F\in L_{p,\frac32}(Q^+),\qquad\forall p\in[1,\infty). 
\end{equation}
Now let  $W$ be a mild solution of  
\begin{align}\label{e.Wequation}
&\partial_{t} W-\Delta W+\nabla P=F-U\cdot\nabla U,\,\,\,\,\,\, \nabla\cdot W=0\,\,\,\mbox{in}\,\,\, \R^3_+\times (-\infty,0),
\end{align}
with no-slip boundary condition $W|_{\partial\mathbb{R}^3_{+}}=0$ and initial data $W(\cdot,-1)=0$. 
Moreover, u\-sing  \eqref{Uenergyclass} and \eqref{e.intF}, together with maximal regularity \cite{GS91},  
we see that $W\in L_{1}(-1,0; L_{\frac{9}{8}}(\mathbb{R}^3_{+}))$. Using this, \eqref{Uenergyclass}, the fact that $U$ has compact support and the Liouville theorem \cite[Theorem 5]{MMP17a} for the linear Stokes system, we see that $U\equiv W$. Namely, $U$ is a mild solution to \eqref{e.Ulocalisedequation} with initial data $U(\cdot,-1)=0$. 

\noindent\textbf{Step 2: zoom in on the singularity and a priori bounds}\\
Let $R_k\downarrow 0$ and consider 
\begin{equation}\label{e.defUk}
U^{(k)}(y,s)=R_k U(R_k y,R_k^2s),\quad P^{(k)}(y,s)=R_k^2 P(R_k y,R_k^2s)
\end{equation}
for $y\in\R^3_+$ and $s\in(-R_{k}^{-2},0)$. 
From \eqref{e.Ulocalisedequation} we have that $(U^{(k)},P^{(k)})$ is a mild solution to
\begin{align}
&\partial_sU^{(k)}+U^{(k)}\cdot\nabla U^{(k)}-\Delta U^{(k)}+\nabla P^{(k)}=R_k^3F(R_ky,R_k^2s),\qquad \nabla\cdot U^{(k)}=0\label{e.eqUk}
\end{align}
in $\mathbb{R}^3_{+}\times (-R_{k}^{-2},0)$. Moreover, 
from the bound \eqref{e.aprioriU} we have
\begin{equation}\label{e.scaledenUk}
\sup_{k}\Big(\sup_{0<r<\frac{r'}{R_{k}}}A(U^{(k)},r)+E(U^{(k)},r)+D_{\frac32}(P^{(k)},r)\Big)<\infty,
\end{equation}
and from \eqref{e.typeIU} we have that 
\begin{equation}\label{e.typeIUk}
\sup_{k}\Big(\sup_{s\in(-R_{k}^{-2},0)}(-s)^\frac12\|U^{(k)}(\cdot,s)\|_{L_\infty(\R^3_+)}\Big)<\infty.
\end{equation}

\noindent\textbf{Step 3: passage to the limit}\\
From the a priori bound \eqref{e.scaledenUk} and \eqref{e.typeIUk}, we have that there exists $(U^{(\infty)},P^{(\infty)})$ defined in $Q^+$ such that
\begin{align*}
&\sup_{0<r<\infty}A(U^{(\infty)},r)+E(U^{(\infty)},r)+D_{\frac32}(P^{(\infty)},r)<\infty,\\
&\sup_{s\in(-\infty,0)}(-s)^\frac12\|U^{(\infty)}(\cdot,s)\|_{L_\infty(\R^3_+)}<\infty.
\end{align*}
We also have in particular the strong convergence \eqref{e.strongcvUkbis} and the weak convergence $P^{(k)}\rightharpoonup P^{(\infty)}$ in $L_\frac32(Q^+(1))$. Moreover since
\begin{equation*}
\big\|R_k^3F(R_ky,R_k^2s)\big\|_{L_{p,\frac32}}\longrightarrow 0\qquad\forall p>\tfrac95,
\end{equation*}
we have that $(U^{(\infty)},P^{(\infty)})$ is a 
weak solution to the Navier-Stokes equations 
\begin{equation*}
\partial_sU^{(\infty)}+U^{(\infty)}\cdot\nabla U^{(\infty)}-\Delta U^{(\infty)}+\nabla P^{(\infty)}=0,\quad \nabla\cdot U^{(\infty)}=0
\end{equation*}
in $Q^+$. Furthermore, using that $U^{(k)}$ is a mild solution on $\mathbb{R}^3_{+}\times (-R_{k}^{-2},0)$ and the same arguments as in the proof of Proposition \ref{2Dreductionbycompactness}, we see that $U^{(\infty)}$ is a mild solution on $Q^{+}$. The Type I bound and the fact that $U^{(\infty)}$ is a mild solution imply that $U^{(\infty)}\in C^\infty(\overline{\R^3_+}\times(-\infty,0))$, see \cite[Theorem 1.3]{barker2015ancient}. 
The same reasoning as for $(U^{(k)},P^{(k)})$ shows that $(u^{(k)},p^{(k)})$ tends to a limit
$(u^{(\infty)},p^{(\infty)})$. By the stability of Type I singularities stated in Lemma \ref{stabilitysingularpointshalfspace}, the space-time point $(0,0)$ is a singular point for $(u^{(\infty)},p^{(\infty)})$. But the Bogovskii corrector $B$ belongs to a subcritical space, see \eqref{e.BdtBbis}, so it tends to zero. Thus, $u^{(\infty)}=U^{(\infty)}$ and therefore $u^{(\infty)}$ is a mild solution, is smooth, in addition to being a suitable weak solution and having a singular point at $(0,0)$. 

The conclusion that $u^{(\infty)}$ is mild and smooth is the whole purpose of the localisation procedure and of the introduction of $U^{(k)}$. We now work with $u^{(k)}$ and $u^{(\infty)}$ only.

\noindent\textbf{Step 4: convergence of the vorticity}\\
Our objective here is to prove that for all $K\in \N\setminus\{0\}$ there is a subsequence such that, 
\begin{equation}\label{e.strongcvvortst}
\int\limits_{-1}^0\Bigg(\int\limits_{B^+(K)}|\nabla u^{(k)}-\nabla u^{(\infty)}|^\frac98dx\Bigg)^\frac43ds\stackrel{k\rightarrow\infty}{\longrightarrow} 0.
\end{equation} 
Let us prove \eqref{e.strongcvvortst}. Let $K\in\mathbb N\setminus\{0\}$. We will take $k(K,r')$ sufficiently large such that $-\frac{r'}{R_k}<-K$ throughout. We first claim that 
\begin{equation}\label{e.maxreg}
M'':=\sup_{k}\|u^{(k)}\|_{W^{2,1}_{\frac98,\frac32}(Q^+(K))}+\|\nabla p^{(k)}\|_{L_{\frac32,\frac98}(Q^+(K))}<\infty.
\end{equation}
This is a consequence of the local maximal regularity for the Stokes system in the half-space (see \cite[Theorem 1.2]{Ser09}). Indeed from \eqref{e.eqUk}, $u^{(k)}$ solves the Stokes system
\begin{align*}
&\partial_su^{(k)}-\Delta u^{(k)}+\nabla p^{(k)}=-u^{(k)}\cdot\nabla u^{(k)},\quad \nabla\cdot u^{(k)}=0\quad\mbox{in}\ Q^+(2K).
\end{align*}
Moreover,  using \eqref{locallyequal}-\eqref{e.aprioriU} we see that for all $k(K,\delta')$ sufficiently large  the scaled energy \eqref{e.scaledenUk} is uniformly bounded in $k$. Hence, we also have
\begin{equation*}
u^{(k)}\cdot\nabla u^{(k)}\in L_{\frac98,\frac32}(Q^+(K))
\end{equation*}
with uniform bounds independent of the sequence parameter $k$. Furthermore, using \eqref{e.scaledenUk}, $p^{(k)}-(p^{(k)})_{B^{+}(K)}$ is uniformly bounded in $L_{\frac{3}{2}}(-K^2,0; L_{\frac{9}{8}}(B^{+}(K)))$. 
This concludes the proof of the bound \eqref{e.maxreg}. It now follows similarly to the proof of Proposition \ref{2Dreductionbycompactness} (Step 2) in Section \ref{sec.four} 
that up to extracting a subsequence
\begin{align}
&u^{(k)}\longrightarrow u^{(\infty)}\qquad\mbox{strongly in}\quad C^0([-1,0];L^\frac98(B^+(K)),\label{e.strongcvUk}\\
\label{e.strongcvUkbis}
&u^{(k)}\longrightarrow u^{(\infty)}\qquad\mbox{strongly in}\quad L^3((-1,0)\times B^+(K)).
\end{align}
We can then extract a diagonal subsequence such that this convergence holds on $B^{+}(K)\times (-1,0)$ for every $K\in\mathbb{N}$. 
To see the strong convergence of the gradient, we rely on Ehrling's inequality \cite[II.5.20]{Galdibook}. Let $\ep>0$ be fixed. Let $M''$ be as defined in \eqref{e.maxreg}. Then there exists a constant $C(\ep,M'')\in(0,\infty)$ 
such that for almost all $t\in(-1,0)$,
\begin{multline*}
\|\nabla u^{(k)}(\cdot,t)-\nabla u^{(\infty)}(\cdot,t)\|_{L^\frac98(B^+(K))}\\
\leq C(\ep,M'')\|u^{(k)}(\cdot,t)-u^{(\infty)}(\cdot,t)\|_{L^\frac98(B^+(K))}+\frac{\ep}{4M''}\|\nabla^2 u^{(k)}(\cdot,t)-\nabla^2 u^{(\infty)}(\cdot,t)\|_{L^\frac98(B^+(K))}.
\end{multline*}
Integrating in time yields
\begin{multline*}
\int\limits_{-1}^0\|\nabla u^{(k)}-\nabla u^{(\infty)}\|_{L^\frac98(B^+(K))}^\frac32ds\\
\leq C(\ep,M'')\int\limits_{-1}^0\|u^{(k)}-u^{(\infty)}\|_{L^\frac98(B^+(K))}^\frac32ds+\frac{\ep}{2M''}\int\limits_{-1}^0\|\nabla^2 u^{(k)}-\nabla^2 u^{(\infty)}\|_{L^\frac98(B^+(K))}^\frac32ds.
\end{multline*}
Hence
\begin{equation*}
\int\limits_{-1}^0\|\nabla u^{(k)}-\nabla u^{(\infty)}\|_{L^\frac98(B^+(K))}^\frac32ds
\leq C(\ep,M'')\int\limits_{-1}^0\|u^{(k)}-u^{(\infty)}\|_{L^\frac98(B^+(K))}^\frac32ds+\frac{\ep}{2},
\end{equation*}
and we finally get from \eqref{e.strongcvUk} for $k$ sufficiently large that
\begin{equation*}
\int\limits_{-1}^0\|\nabla u^{(k)}-\nabla u^{(\infty)}\|_{L^\frac98(B^+(K))}^\frac32ds\leq \frac\ep2+\frac\ep2=\ep,
\end{equation*}
which implies the convergence \eqref{e.strongcvvortst}.

\noindent\textbf{Step 5: obtaining the contradiction via the use of Egoroff's theorem}\\
As in Section \ref{sec.four} (Steps 4 and 5) we need to discuss between the case when the vorticity vanishes identically on the boundary or not. Here again two things can happen: 
\begin{itemize}
\item either the vorticity vanishes identically on the boundary for all times $t\in(-1,0)$, i.e.
\begin{equation*}
\omega^{(\infty)}(x',0,t)=0\qquad\forall x'\in\R^2,\ t\in(-1,0)
\end{equation*}
in which case we conclude that $\omega^{(\infty)}$ vanishes identically on $\R^3_+\times(-1,0)$ as in the proof of Theorem \ref{CAnconcentratingglobal} (Step 5) in Section \ref{sec.four},
\item or there exists a time $t_0\in(-1,0)$ and ${x_0}'\in\R^2$ such that $\omega^{(\infty)}({x_0}',0,t_0)\neq 0$, in which case there exist a whole interval $I_{t_0}\subset (-1,0)$ around $t_0$ where the vorticity does not vanish on the boundary. \emph{In what follows below, we consider only this case.}
\end{itemize}
Let $t_0$ and $I_{t_0}$ be fixed as above. We deduce from \eqref{e.strongcvvortst} the convergence of the vorticity in $L_{\frac98,\frac32}(B^+(K)\times(-1,0))$. Indeed, for all $K\in\N\setminus\{0\}$, for almost every $t\in I_{t_0}$
\begin{equation*}
\nabla\times u^{(k)}(\cdot,t)\longrightarrow \nabla\times u^{(\infty)}(\cdot,t)\qquad \mbox{strongly in}\quad L^\frac98(B^+(K)).
\end{equation*}
Therefore, there exists $t_1\in I_{t_0}$ such that for all $K\in\N\setminus\{0\}$
\begin{equation}\label{e.strongcvvort}
\nabla\times u^{(k)}(\cdot,t_1)\longrightarrow \nabla\times u^{(\infty)}(\cdot,t_1)\qquad \mbox{strongly in}\quad L^\frac98(B^+(K)).
\end{equation}
This convergence is the substitute to the convergence of the vorticity in $C_{loc}(Q^+)$, and as we will see in Step 5 below, it is enough to carry out the proof of Theorem \ref{theo.locgeomcrit}. 

Our final goal is to see that the limit flow at time $t_1$ is two-dimensional in the sense that $u^{(\infty)}(y,t_{1})=(0,u^{(\infty)}_2(y_2,y_3,t_{1}),u^{(\infty)}_3(y_2,y_3,t_{1}))$. We now rely on the convergence of the vorticity \eqref{e.strongcvvort}, to show that the vorticity direction for $u^{(\infty)}$ is constant. We denote by $\xi^{(\infty)}$ the vorticity direction of $u^{(\infty)}$. We have that 
\begin{align*}
S^+_0=\ &\big\{y\in\R^3_+,\, \infty>|\nabla\times u^{(\infty)}(\cdot,t_{1})|>0\big\}\\
=\ &\bigcup_{K\in\N\setminus\{0\}}\left(\big\{y\in\R^3_+,\, \infty>|\nabla\times U^{(\infty)}(\cdot,t_{1})|>K^{-1}\big\}\cap B^+(K)\right).
\end{align*}
Let us denote 
\begin{equation*}
S^+_K:=\big\{y\in\R^3_+,\, \infty>|\nabla\times U^{(\infty)}(\cdot,t_{1})|>K^{-1}\big\}\cap B^+(K).
\end{equation*}
For $K\in\mathbb N\setminus\{0\}$, the convergence \eqref{e.strongcvvort} implies that 
\begin{equation*}
\nabla\times u^{(k)}(\cdot,t_1)\longrightarrow \nabla\times u^{(\infty)}(\cdot,t_1)\quad\mbox{a.e. in}\ B^+(K).
\end{equation*}
Let $\ep>0$. By Egoroff's theorem \cite[Theorem 1.16]{EGbook}, there exists $N_{\ep,K}$ such that
\begin{align*}
&N_{\ep,K}\subset S^+_K,\qquad |N_{\ep,K}|\leq\ep,
\end{align*}
and 
\begin{equation}\label{e.unifcvego}
\nabla\times u^{(k)}(\cdot,t_1)\longrightarrow \nabla\times u^{(\infty)}(\cdot,t_1)\quad\mbox{uniformly on}\quad S^+_{K}\setminus N_{\ep,K}.
\end{equation}
There exists $k_0(K,d)$ such that for all $k\geq k_0$, we have $\frac{1}{KR_k}>d$. Hence, by the vorticity alignment condition \eqref{vorticitydirectionloc} for $u$ we have for $k\geq k_0$, for all $x,y\in S^+_K$,
\begin{equation}\label{e.vortaliXik}
\left|\xi^{(k)}(x,t_0)-\xi^{(k)}(y,t_0)\right|\leq\eta(R_k|x-y|).
\end{equation}
Therefore, for all $x,y\in S^+_K\setminus N_{\ep,K}$, the uniform convergence \eqref{e.unifcvego} implies passing to the limit in \eqref{e.vortaliXik}, we get that
\begin{equation*}
\left|\xi^{(\infty)}(x,t_1)-\xi^{(\infty)}(y,t_1)\right|=0.
\end{equation*}
This means that the vorticity direction $\xi^{(\infty)}$ is constant on $S^+_K\setminus N_{\ep,K}$. Hence, for $\varepsilon_0$ sufficiently small, the vorticity direction is constant on $$\cup_{\varepsilon<\varepsilon_{0}}S^+_K\setminus N_{\ep,K}=S^+_K\setminus\cap_{\varepsilon<\varepsilon_{0}} N_{\ep,K}.$$ Since $\left|\cap_{\ep}N_{\ep,K}\right|=0$, we have $\xi^{(\infty)}$ is constant almost everywhere, and also everywhere because $u^{(\infty)}$ is smooth on $\overline{\R^3_+}\times(-\infty,0)$ as stated at the end of Step 3. We conclude as in Step 4 using Proposition \ref{2Dreductionbycompactness}.

\begin{remark}[vorticity alignment on concentrating sets, local setting]\label{rem.vortalconsets}
Theorem \ref{CAnconcentratingglobal} can also be localised, in the same way we localised Theorem 1.3 of \cite{giga2014liouville} in our Theorem \ref{theo.locgeomcrit}. 
 The conti\-nuous alignment condition has to be assumed on 
$$ \mathcal{C}_{\delta, x_{0}}:= \bigcup_{t\in(T-1,T)}\{(x,t): |x-x_{0}|< \delta \sqrt{T-t}\}.$$
\end{remark}

\subsection*{Funding and conflict of interest.} 
The second author is partially supported by the project BORDS grant ANR-16-CE40-0027-01 and by the project SingFlows grant ANR-18-CE40-0027 of the French National Research Agency (ANR). 
The second author also acknow\-ledges financial support from the IDEX of the University of Bordeaux for the BOLIDE project. The authors declare that they have no conflict of interest. 

\small
\bibliographystyle{abbrv}
\bibliography{concentration.bib}

\end{document}